\documentclass[12pt]{article}
\usepackage{enumerate}
\usepackage{amsthm,amssymb,amsmath,mathrsfs}
\usepackage{enumitem,mathtools}
\usepackage{url}
\usepackage{fullpage}
\usepackage{hyperref}
\usepackage{graphicx}
\usepackage{array}
\usepackage{multirow}
\usepackage{longtable}
\usepackage{tikz}
\usepackage{bbm}
\usepackage{bm}
\usepackage{pifont}
\usepackage{wasysym}
\usepackage{tensor}
\usepackage{caption}
\usepackage{subcaption}
\usepackage{hhline}
\usepackage{tabularx}

\makeatletter
\newcommand{\myfnsymbol}[1]{%
  \expandafter\@myfnsymbol\csname c@#1\endcsname
}
\newcommand{\@myfnsymbol}[1]{%
  \ifcase #1
  \or 
  \or 1
  \or $*$
  \or $\dagger$
  \or $\ddagger$
  \fi
}
\newcommand{\spmsntu}{\@myfnsymbol{2}}
\newcommand{\emailA}{\@myfnsymbol{3}}
\newcommand{\emailB}{\@myfnsymbol{4}}
\newcommand{\emailC}{\@myfnsymbol{5}}
\makeatother

\makeatletter
\newcommand*{\transpose}{%
  {\mathpalette\@transpose{}}%
}
\newcommand*{\@transpose}[2]{%
  \raisebox{\depth}{$\m@th#1\intercal$}%
}

\newtheorem{thm}{Theorem}[section]
\newtheorem{conj}[thm]{Conjecture}
\newtheorem{cor}[thm]{Corollary}
\newtheorem{lem}[thm]{Lemma}
\newtheorem{prop}[thm]{Proposition}

\theoremstyle{definition}
\newtheorem{dfn}[thm]{Definition}

\newtheorem{ex}[thm]{Example}
\newtheorem{qn}[thm]{Question}

\DeclareRobustCommand{\stirling}{\genfrac\{\}{0pt}{}}

\makeatletter
\DeclareRobustCommand{\veeplus}{\mathbin{\mathpalette\veeplus@@\relax}}
\newcommand{\veeplus@@}[2]{%
  \vbox{\offinterlineskip
    \sbox\z@{$\m@th#1\vee$}%
    \ialign{%
      \hfil##\hfil\cr
      $\m@th#1{}_{+}\kern-\scriptspace$\cr
      \noalign{\kern-.\ht\z@}
      \box\z@\cr
    }%
  }%
}
\DeclareRobustCommand{\veeminus}{\mathbin{\mathpalette\veeminus@@\relax}}
\newcommand{\veeminus@@}[2]{%
  \vbox{\offinterlineskip
    \sbox\z@{$\m@th#1\vee$}%
    \ialign{%
      \hfil##\hfil\cr
      $\m@th#1{}_{-}\kern-\scriptspace$\cr
      \noalign{\kern-.\ht\z@}
      \box\z@\cr
    }%
  }%
}
\makeatother

\AtBeginDocument{%
   \def\MR#1{}
}

\title{Chromatic polynomials of signed graphs and \\
dominating-vertex deletion formulae}
\author{Gary R.W. Greaves \and Jeven Syatriadi \and Charissa I. Utomo}

\date{}

\begin{document}

\renewcommand{\thefootnote}{\myfnsymbol{footnote}}

\maketitle

\footnotetext[1]{Division of Mathematical Sciences, School of Physical and Mathematical Sciences, Nanyang Technological University, 21 Nanyang Link, Singapore 637371.}%
\footnotetext[1]{\textit{Email addresses}: \texttt{gary@ntu.edu.sg} (G.R.W. Greaves), \texttt{jeve0002@e.ntu.edu.sg} (J. Syatriadi), \texttt{cutomo001@e.ntu.edu.sg} (C.I. Utomo).}%

\setcounter{footnote}{0}
\renewcommand{\thefootnote}{\arabic{footnote}}

\begin{abstract}
    We exhibit non-switching-isomorphic signed graphs that share a common underlying graph and common chromatic polynomials, thereby answering a question posed by Zaslavsky.
    For various joins of all-positive or all-negative signed complete graphs, we derive a closed-form expression for their chromatic polynomials.
    As a generalisation of the chromatic polynomials for a signed graph, we introduce a new pair of bivariate chromatic polynomials.
    We establish recursive dominating-vertex deletion formulae for these bivariate chromatic polynomials.
    Finally, we show that for certain families of signed threshold graphs, isomorphism is equivalent to the equality of bivariate chromatic polynomials.
\end{abstract}

\section{Introduction}
\label{sec:intro}

\subsection{Background}

The chromatic polynomial of a graph is considered to be one of the most important polynomials in graph theory, with its origins stemming from the Four-Colour Theorem.
See the expository paper of Read~\cite{Read68} for an introduction.
As with many other notions for graphs, the chromatic polynomial, which counts the number of ways one can properly colour a graph with $\lambda$ colours, has been generalised to signed graphs.
We refer to the survey of Steffen and Vogel~\cite{SV21} for the history of signed graph colouring. 

A \textbf{signed graph} $\Sigma = (\Gamma,\sigma)$ is a graph $\Gamma = \left( V(\Gamma), E(\Gamma) \right)$ equipped with \textbf{signature} function $\sigma : E(\Gamma) \to \{\pm 1\}$, which assigns each edge of $\Gamma$ a positive or negative sign.
An edge $e \in E(\Gamma)$ is called a \textbf{positive edge} if $\sigma(e) = 1$ and a \textbf{negative edge} if $\sigma(e) = -1$.
In our figures below, solid edges represent positive edges while dashed edges represent negative edges.
The graph $\Gamma$ is referred to as the \textbf{underlying graph} of $\Sigma$ and denoted by $\left\vert \Sigma \right\vert$.
Throughout this paper, the underlying graphs are \emph{finite simple graphs}.

Zaslavsky~\cite{ZasCol} introduced the notion of \textit{colouring} for signed graphs as follows: for a fixed nonnegative integer $k$, a function $\kappa : V(\Gamma) \to \{-k,\dots,-1,0,1,\dots,k\}$ such that $\kappa(v) \ne \sigma(\{v,w\})\kappa(w)$ for all edges $\{v,w\} \in E(\Gamma)$, is called a \textbf{proper colouring} of $\Sigma$ in $2k+1$ signed colours; and a function $\kappa : V(\Gamma) \to \{-k,\dots,-1\}\cup\{1,\dots,k\}$ such that $\kappa(v) \ne \sigma(\{v,w\})\kappa(w)$ for all edges $\{v,w\} \in E(\Gamma)$, is called a \textbf{proper zero-free colouring} of $\Sigma$ in $2k$ signed colours.
If $\lambda=2k+1$ for some nonnegative integer $k$, then the number of proper colourings of $\Sigma$ in $\lambda$ signed colours is equal to the value of a polynomial, which we denote by $\mathsf O(\Sigma, x)$, when evaluated at $\lambda$.
Likewise, if $\lambda=2k$ for some nonnegative integer $k$, then the number of proper zero-free colourings of $\Sigma$ in $\lambda$ signed colours is equal to the value of a polynomial, which we denote by $\mathsf E(\Sigma, x)$, when evaluated at $\lambda$.
We refer to these two polynomials together as the \emph{chromatic polynomials} of $\Sigma$.
These polynomials can be realised as the constituents of a certain Ehrhart quasipolynomial~\cite[Theorem 5.6]{BZ06} and thus, they are usually referred to as a pair by the term \emph{chromatic quasipolynomial}.

Zaslavsky~\cite{ZasSign} showed that, up to switching isomorphism, there are precisely six signed graphs whose underlying graph is the Petersen graph.
Beck et al.~\cite{Involve} computed the chromatic polynomials of each of these six signed graphs.
Ren et al.~\cite{RQHZ22} studied the polynomial $\mathsf O(\Sigma, x)$ for some families of signed graphs.
Beck and Hardin~\cite{BH15} generalised the chromatic polynomial of a signed graph to a bivariate polynomial analogously to the Dohmen-P\" onitz-Tittmann bivariate chromatic polynomial of an (unsigned) graph \cite{DPT03}.

\subsection{Main results and organisation}

Our first contribution is to answer the following question due to Zaslavsky~\cite[Question 9.1]{Zas12}.
We will define the notion of isomorphism and switching for signed graphs in Section~\ref{sec:signedchrom}.

\begin{qn}
\label{qn:zas}
    Do there exist non-switching-isomorphic signed graphs $\Sigma_1$ and $\Sigma_2$ such that $\left\vert \Sigma_1 \right\vert = \left\vert \Sigma_2 \right\vert$ and $\mathsf E(\Sigma_1, x) = \mathsf E(\Sigma_2, x)$ or $\mathsf O(\Sigma_1, x) = \mathsf O(\Sigma_2, x)$?
\end{qn}

Zaslavsky~\cite{Zas12} showed that the answer to Question~\ref{qn:zas} is no if the shared underlying graph is $2$-regular.
In Section~\ref{sec:signedchrom}, we answer Question~\ref{qn:zas} in the affirmative and provide an example of signed graphs $\Sigma_1$ and $\Sigma_2$ such that $\left\vert \Sigma_1 \right\vert = \left\vert \Sigma_2 \right\vert$, $\mathsf E(\Sigma_1, x) = \mathsf E(\Sigma_2, x)$, and $\mathsf O(\Sigma_1, x) = \mathsf O(\Sigma_2, x)$.
We also prove that for any signed graph $\Sigma$, if $\mathsf E(\Sigma, x) = \mathsf O(\Sigma, x)$ then $\Sigma$ must be balanced.

Beck et al.~\cite{Involve} computed the chromatic polynomials of signed graphs whose underlying graphs are the complete graphs or the Petersen graph.
However, we find some inaccuracies in their computations of $\mathsf E(\Sigma, x)$, which we rectify in Table~\ref{tab:involvePetersen} and Table~\ref{tab:involvecomplete}.

In \cite[Section 8.9]{ZasChromInv}, Zaslavsky exhibited a closed formula for the chromatic polynomials of $-K_n$.
As an extension of his work, in Section~\ref{sec:signedjoin}, we exhibit closed-form expressions for the chromatic polynomials of various joins of all-positive or all-negative signed complete graphs.

It is well known that the chromatic polynomial of an (unsigned) graph satisfies a recursive edge deletion-contraction rule.
This rule also applies to signed graphs, taking into account parallel edges and loops (see \cite[Theorem 2.3]{ZasCol} or \cite[Proposition 2]{Categ22}).
In Section~\ref{sec:bivar}, we derive recursive formulae for the deletion of a positive dominating vertex or negative dominating vertex of a signed graph.
The recursive dominating-vertex deletion formulae require us to define the \emph{bivariate chromatic polynomials} of signed graphs, which is a generalisation of the chromatic polynomials of signed graphs.
Compared with the chromatic polynomials of signed graphs, the bivariate chromatic polynomials are not a switching invariant but they are a stronger isomorphism invariant.
We point out that the definition of the bivariate chromatic polynomials in this paper differs from that of Beck and Hardin~\cite{BH15}.
Lastly, we show that for a certain family of \emph{signed threshold graphs}, each isomorphism class can be uniquely represented by the corresponding bivariate chromatic polynomials.

\subsection{Notation and preliminaries}

Let $\Sigma_1 = (\Gamma_1,\sigma_1)$ and $\Sigma_2 = (\Gamma_2,\sigma_2)$ be signed graphs.
Then $\Sigma_1$ and $\Sigma_2$ are \textbf{isomorphic}, denoted by $\Sigma_1 \cong \Sigma_2$, if there exists a graph isomorphism $\varphi : V(\Gamma_1) \to V(\Gamma_2)$ such that for all edges $e=\{v,w\} \in E(\Gamma_1)$, we have $\sigma_1(e) = \sigma_2 (e^{\varphi})$ where $e^{\varphi} = \{ \varphi(v), \varphi(w) \} \in E(\Gamma_2)$.
In other words, two signed graphs are isomorphic if there exists a graph isomorphism between them that preserves the signs of the edges.

Next, we define the notion of \textbf{switching}.
Let $\Sigma = (\Gamma, \sigma)$ be a signed graph and let $X \subseteq V(\Gamma)$.
Switching $X$ means inverting the sign of all edges that are incident with precisely one vertex in $X$.
If $X=\{v\}$ then switching $v$ inverts the signs of all edges incident with $v$.
Two signed graphs $\Sigma_1 = (\Gamma,\sigma_1)$ and $\Sigma_2 = (\Gamma,\sigma_2)$ are \textbf{switching equivalent}, denoted by $\Sigma_1 \sim \Sigma_2$, if there exists $X \subseteq V(\Gamma)$ such that one can be obtained from the other by switching $X$.

Lastly, two signed graphs $\Sigma_1$ and $\Sigma_2$ are \textbf{switching isomorphic}, denoted by $\Sigma_1 \simeq \Sigma_2$, if there exists a signed graph $\Sigma$ such that $\Sigma_1 \cong \Sigma$ and $\Sigma \sim \Sigma_2$.
In other words, two signed graphs $\Sigma_1$ and $\Sigma_2$ are switching isomorphic if $\Sigma_1$ is isomorphic to a switching of $\Sigma_2$.
Clearly, if $\Sigma_1 \cong \Sigma_2$ or $\Sigma_1 \sim \Sigma_2$, then $\Sigma_1 \simeq \Sigma_2$.

Each relation $\cong$, $\sim$, and $\simeq$ is an equivalence relation and we call the corresponding equivalence classes \textbf{isomorphism classes}, \textbf{switching equivalence classes}, and \textbf{switching isomorphism classes}, respectively.
The isomorphism classes and the switching isomorphism classes can be represented geometrically by unlabeled signed graphs.
If $\Sigma$ is an unlabeled signed graph, then $\left\vert \Sigma \right\vert$ denotes the unlabeled underlying graph of $\Sigma$.
Moreover, it should be clear from the context whether we work with specific signed graphs or with the isomorphism classes of signed graphs instead.

Let $\Sigma = (\Gamma, \sigma)$ be a signed graph.
We call $\Sigma$ \textbf{all-positive} if $\sigma(e) = 1$ for all $e \in E(\Gamma)$ and \textbf{all-negative} if $\sigma(e) = -1$ for all $e \in E(\Gamma)$.
Let $+\Sigma$ denote the all-positive signed graph with underlying graph $\Gamma$, and let $-\Sigma$ denote the all-negative signed graph with underlying graph $\Gamma$.
Let $\Sigma^+$ denote the all-positive signed graph where $V(\left\vert \Sigma^+ \right\vert) = V(\Gamma)$ and $E(\left\vert \Sigma^+ \right\vert) = \sigma^{-1}(1)$.
Similarly, let $\Sigma^-$ denote the all-negative signed graph where $V(\left\vert \Sigma^- \right\vert) = V(\Gamma)$ and $E(\left\vert \Sigma^- \right\vert) = \sigma^{-1}(-1)$.
A \textbf{connected component} $\Lambda$ of $\Sigma$ is a signed graph where $\left\vert \Lambda \right\vert$ is a connected component of $\Gamma$ and the signature is $\sigma$ restricted to the edges of $\left\vert \Lambda \right\vert$.
For $Y \subseteq E(\Gamma)$, we denote by $\Sigma \vert Y$ the signed spanning subgraph $\left( (V(\Gamma),Y), \sigma \vert_Y \right)$, where $\sigma \vert_Y$ denotes the restriction of $\sigma$ to $Y$.

A \textbf{signed complete graph} $\Sigma$ is a signed graph such that $\left\vert \Sigma \right\vert$ is an (unsigned) complete graph.
Let $n \geqslant 0$ be an integer.
We denote by $+K_n$ (resp.\ $-K_n$) the all-positive (resp.\ all-negative) signed complete graph of order $n$.
The vertexless and edgeless signed graph is denoted by $K_0$ and the edgeless signed graph with a single vertex is denoted by $K_1$.
Hence, $K_0 = +K_0 = -K_0$ and $K_1 = +K_1 = -K_1$.

A \textbf{cycle} of $\Sigma$ is a connected $2$-regular subgraph $\mathcal C$ of $|\Sigma|$ together with $\sigma$ restricted to the edges of $\mathcal C$.
We call a cycle of $\Sigma$ a \textbf{negative cycle} if it has an odd number of negative edges.
We call $\Sigma$ \textbf{unbalanced} if it contains a negative cycle and \textbf{balanced} otherwise.
We have that $\Sigma$ is balanced if and only if $+\Sigma \sim \Sigma$, see \cite[Corollary 1.3]{SV21} or \cite[Corollary 3.3]{ZasSign}.

\section{Chromatic polynomials of a signed graph}
\label{sec:signedchrom}

Let $\Gamma = \left( V(\Gamma), E(\Gamma) \right)$ be an unsigned graph.
Let $\lambda \geqslant 0$ be an integer and let $C$ be a finite set of size $\lambda$.
Define $f(\Gamma, \lambda)$ to be the number of functions $\kappa : V(\Gamma) \to C$ such that $\kappa(v) \neq \kappa(w)$ for all edges $\{v, w\} \in E(\Gamma)$.
We remark that $f(\Gamma, \lambda)$ depends only on $\Gamma$ and $\lambda$; it does not depend on the choice of $C$.
Then there exists a unique polynomial $\chi(\Gamma, x) \in \mathbb{Z}[x]$ such that for all integers $\lambda \geqslant 0$, we have $f(\Gamma, \lambda) = \chi(\Gamma, \lambda)$.
We call $\chi(\Gamma, x)$ the \textbf{chromatic polynomial} of $\Gamma$ \cite{Read68}.

Generalising the chromatic polynomial to signed graphs, Zaslavsky~\cite{ZasCol} introduced vertex-colouring for signed graphs.
In this paper, we adopt a slight modification of the signed graph colouring by Zaslavsky.
Our modification allows us to simplify some proofs in this paper and to accommodate a generalisation of signed graph colouring, which we introduce in Section~\ref{sec:bivar}.

\begin{dfn}
\label{dfn:properCcolouring}
    Let $\Sigma = (\Gamma, \sigma)$ be a signed graph and let $C$ be a finite subset of $\mathbb{Z}$.
    A \textbf{proper $C$-colouring} of $\Sigma$ is a function $\kappa : V(\Gamma) \to C$ such that $\kappa(v) \neq \sigma(\{v,w\}) \kappa(w)$ for all edges $\{v,w\} \in E(\Gamma)$.
\end{dfn}

Let $\lambda \geqslant 0$ be an integer and let $\mathbb{Z}^*$ be the set of nonzero integers.
We define 
\[
C_\lambda \coloneqq \begin{cases}
    \displaystyle \mathbb{Z}^* \cap \left[-\frac{\lambda}{2}, \frac{\lambda}{2} \right], & \text{ if $\lambda$ is even}; \\
    \displaystyle \mathbb{Z} \cap \left[-\frac{\lambda-1}{2}, \frac{\lambda-1}{2} \right], & \text{ if $\lambda$ is odd}.
\end{cases}
\]

In either case, note that $\left\vert C_{\lambda} \right\vert = \lambda$ and we call $C_{\lambda}$ the \textbf{$\lambda$-colour set}.
The definition of $C_{\lambda}$ here is the same as the definition of the colour set in \cite{ZasCol} and the set $M_n$ in \cite[p.\ 2]{MRS16}.

\begin{dfn}
\label{dfn:signedchrom}
    Let $\Sigma$ be a signed graph.
    Let $\lambda \geqslant 0$ be an integer and let $C_{\lambda}$ be the $\lambda$-colour set.
    Define $f(\Sigma, \lambda)$ to be the number of proper $C_{\lambda}$-colourings of $\Sigma$.
\end{dfn}

Denote by $c(\Sigma)$ the number of connected components of $\Sigma$ and $b(\Sigma)$ the number of balanced connected components of $\Sigma$.
We can determine $f(\Sigma, \lambda)$ by using Theorem~\ref{thm:iechrom} below.
As is standard, we define $0^0 = 1$.

\begin{thm}[{\cite[Propositions 4 and 5]{Categ22}}]
\label{thm:iechrom}
    Let $\Sigma = (\Gamma,\sigma)$ be a signed graph and let $\lambda \geqslant 0$ be an integer.
    Then
        \begin{equation*}
    f(\Sigma, \lambda) = 
        \displaystyle \sum_{i=0}^{|E(\Gamma)|} (-1)^i \sum_{Y \subseteq E(\Gamma), \, |Y|=i} \lambda^{b(\Sigma \vert Y)} \cdot \delta^{c(\Sigma \vert Y)-b(\Sigma \vert Y)},
    \end{equation*}
    where $\delta \in \{0,1\}$ is congruent to $\lambda$ modulo $2$.
\end{thm}

By Theorem~\ref{thm:iechrom}, we conclude that there exist unique $\mathsf E(\Sigma, x)$, $\mathsf O(\Sigma, x) \in \mathbb{Z}[x]$ such that for all integers $\lambda \geqslant 0$, we have $f(\Sigma, \lambda) = \mathsf E(\Sigma, \lambda)$ if $\lambda$ is even, and $f(\Sigma, \lambda) = \mathsf O(\Sigma, \lambda)$ if $\lambda$ is odd.
We call $\mathsf E(\Sigma, x)$ the \textbf{even chromatic polynomial} of $\Sigma$ and $\mathsf O(\Sigma, x)$ the \textbf{odd chromatic polynomial} of $\Sigma$.
For any signed graph $\Sigma$, define $\mathscr C(\Sigma, x) \coloneqq \left( \mathsf E(\Sigma, x), \; \mathsf O(\Sigma, x) \right)$, which we refer to as the \textbf{chromatic polynomials} of $\Sigma$.
Given signed graphs $\Sigma_1$ and $\Sigma_2$, it follows that $\mathscr C(\Sigma_1, x) = \mathscr C(\Sigma_2, x)$ if and only if $\mathsf E(\Sigma_1, x) = \mathsf E(\Sigma_2, x)$ and $\mathsf O(\Sigma_1, x) = \mathsf O(\Sigma_2, x)$.
Equivalently, we also have that $\mathscr C(\Sigma_1, x) = \mathscr C(\Sigma_2, x)$ if and only if for all integers $\lambda \geqslant 0$, $f(\Sigma_1, \lambda) = f(\Sigma_2, \lambda)$.

For the readers' convenience, below we provide a short proof of a well-known fact that switching isomorphic signed graphs have identical chromatic polynomials.

\begin{cor}
\label{cor:switchisomchrom}
    Let $\Sigma_1$ and $\Sigma_2$ be signed graphs.
    If $\Sigma_1 \simeq \Sigma_2$ then $\mathscr C(\Sigma_1, x) = \mathscr C(\Sigma_2, x)$.
\end{cor}

\begin{proof}
    Suppose that $\Sigma_1$ and $\Sigma_2$ are switching isomorphic.
    By definition, there exists a signed graph $\Sigma$ such that $\Sigma_1 \cong \Sigma$ and $\Sigma \sim \Sigma_2$.
    Two isomorphic signed graphs have identical set of signed spanning subgraphs.
    By Theorem~\ref{thm:iechrom}, we have $f(\Sigma_1, \lambda) = f(\Sigma, \lambda)$ for all integers $\lambda \geqslant 0$.
    Hence, we obtain $\mathscr C(\Sigma_1, x) = \mathscr C(\Sigma, x)$.
    
    Next, we have that $\Sigma = (\Gamma, \sigma)$ and $\Sigma_2 = (\Gamma, \sigma_2)$ are switching equivalent.
    Let $Y \subseteq E(\Gamma)$ and we also have that $\Sigma \vert Y$ and $\Sigma_2 \vert Y$ are switching equivalent.
    Since the number of balanced components does not change under switching, we have $b(\Sigma \vert Y) = b(\Sigma_2 \vert Y)$.
    By Theorem~\ref{thm:iechrom}, we have $f(\Sigma, \lambda) = f(\Sigma_2, \lambda)$ for all integers $\lambda \geqslant 0$.
    Therefore, we conclude that $\mathscr C(\Sigma_1, x) = \mathscr C(\Sigma, x) = \mathscr C(\Sigma_2, x)$.
\end{proof}

Let $\Sigma = (\Gamma, \sigma)$ be a signed graph.
By definition, we have $f(\Gamma, \lambda) = f(+\Sigma, \lambda)$ for all integers $\lambda \geqslant 0$.
This implies that $\chi(\Gamma, x) = \mathsf E(+\Sigma, x) = \mathsf O(+\Sigma, x)$.
Moreover, if $\Sigma$ is balanced, then $+\Sigma \sim \Sigma$, which further implies that $\mathscr C(\Sigma, x) = \mathscr C(+\Sigma, x) = \left( \chi(\Gamma, x), \; \chi(\Gamma, x) \right)$.
The following definition is related to the converse of Corollary~\ref{cor:switchisomchrom} above.

\begin{dfn}
\label{dfn:cochromatic}
    Let $\Sigma_1$ and $\Sigma_2$ be signed graphs.
    Then we call $\Sigma_1$ and $\Sigma_2$ \textbf{co-chromatic} if $\Sigma_1$ and $\Sigma_2$ are not switching isomorphic and $\mathscr C(\Sigma_1, x) = \mathscr C(\Sigma_2, x)$.
\end{dfn}

We will look at some examples of co-chromatic signed graphs.
Before that, first we prove the next theorem, which asserts that $\Sigma$ is balanced precisely when $\mathsf E(\Sigma, x) = \mathsf O(\Sigma, x)$.

\begin{thm}
\label{thm:polyiffbal}
    Let $\Sigma$ be a signed graph.
    Then $\Sigma$ is balanced if and only if $\mathsf E(\Sigma, x) = \mathsf O(\Sigma, x)$.
\end{thm}
\begin{proof}
    If $\Sigma$ is balanced, then we have $\mathsf E(\Sigma, x) = \mathsf O(\Sigma, x) = \chi(|\Sigma|, x)$.
    Otherwise, suppose that $\Sigma = (\Gamma, \sigma)$ is unbalanced.
    Let $d$ be the minimum size of a negative cycle of $\Sigma$ and suppose that there are $m > 0$ such cycles in $\Sigma$.
    Let $Y$ be the edge set of one such cycle in $\Sigma$. 
    Then we have $|Y|=d$ and $b(\Sigma \vert Y) = c(\Sigma \vert Y)-1 = |V(\Gamma)|-d$.
    Consider the polynomial $\mathsf O(\Sigma, x) - \mathsf E(\Sigma, x)$.
    Applying Theorem~\ref{thm:iechrom}, we deduce that the coefficient of $x^{|V(\Gamma)|-d}$ in $\mathsf O(\Sigma, x) - \mathsf E(\Sigma, x)$ is equal to $(-1)^d \cdot m \neq 0$ and therefore, $\mathsf E(\Sigma, x) \neq \mathsf O(\Sigma, x)$.
\end{proof}

Let $\Sigma = (\Gamma, \sigma)$ be a signed graph such that $V(\Gamma) \neq \varnothing$ and $\Gamma$ is connected.
Balance can also be characterised in the following way.
By \cite[Theorem 6.1]{RQHZ22} \footnote{As clarified by Ren et al.~\cite{renetalPersonalCommunication}, Theorem 6.1 in \cite{RQHZ22} holds only for the case where $|\Sigma|$ is connected.}, we have that $\Sigma$ is balanced if and only if $\mathsf O(\Sigma, 0) = 0$.

Theorem~\ref{thm:polyiffbal} implies that co-chromatic signed graphs $\Sigma_1$ and $\Sigma_2$ are either both balanced or both unbalanced.
If both are balanced, then $|\Sigma_1|$ and $|\Sigma_2|$ cannot be isomorphic since $\Sigma_1$ and $\Sigma_2$ are not switching isomorphic.
The case where co-chromatic signed graphs are both balanced is certainly possible since there exist examples of non-isomorphic unsigned graphs that have the same chromatic polynomial \cite{Bari, Read68}.

Zaslavsky~\cite[Question $9.1$]{Zas12} (see Question~\ref{qn:zas}) asked if there are co-chromatic signed graphs $\Sigma_1$ and $\Sigma_2$ where $|\Sigma_1| = |\Sigma_2|$.
If so, then $\Sigma_1$ and $\Sigma_2$ must both be unbalanced.
In Figure~\ref{fig:signedgem}, we provide an example of co-chromatic signed gem graphs $G_1$ and $G_2$, thereby answering Question~\ref{qn:zas} in the affirmative.
We have that
\[
\mathscr C(G_1, x) = \mathscr C(G_2, x) = \left( x (x-2)^2 (x^2 - 3x +3), \; (x-1)^3 (x-2)^2 \right).
\]

\begin{figure}[htbp]
    \centering
    \begin{subfigure}{0.33\textwidth}
        \centering
        \begin{tikzpicture}[scale=1.2, every node/.style={scale=0.9}]
        \tikzstyle{vertex}=[circle, thin, fill=black!90, inner sep=0pt, minimum width=4pt]
        \node [vertex] (v1) at (0,0) {};
        \node [vertex] (v2) at (0:1.75) {};
        \node [vertex] (v3) at (60:1.75) {};
        \node [vertex] (v4) at (120:1.75) {};
        \node [vertex] (v5) at (180:1.75) {};
        \draw  (v1) edge (v2);
        \draw  (v1) edge (v3);
        \draw  (v1) edge (v4);
        \draw  (v1) edge (v5);
        \draw  (v2) edge (v3);
        \draw  (v3) edge (v4);
        \draw  (v4) edge [dashed] (v5);
        \end{tikzpicture}
        \caption*{$G_1$}
    \end{subfigure}
    \begin{subfigure}{0.33\textwidth}
        \centering
        \begin{tikzpicture}[scale=1.2, every node/.style={scale=0.9}]
        \tikzstyle{vertex}=[circle, thin, fill=black!90, inner sep=0pt, minimum width=4pt]
        \node [vertex] (v1) at (0,0) {};
        \node [vertex] (v2) at (0:1.75) {};
        \node [vertex] (v3) at (60:1.75) {};
        \node [vertex] (v4) at (120:1.75) {};
        \node [vertex] (v5) at (180:1.75) {};
        \draw  (v1) edge (v2);
        \draw  (v1) edge (v3);
        \draw  (v1) edge (v4);
        \draw  (v1) edge (v5);
        \draw  (v2) edge (v3);
        \draw  (v3) edge [dashed] (v4);
        \draw  (v4) edge (v5);
        \end{tikzpicture}
        \caption*{$G_2$}
    \end{subfigure}
    \caption{Co-chromatic signed gem graphs $G_1$ and $G_2$.}
    \label{fig:signedgem}
\end{figure}
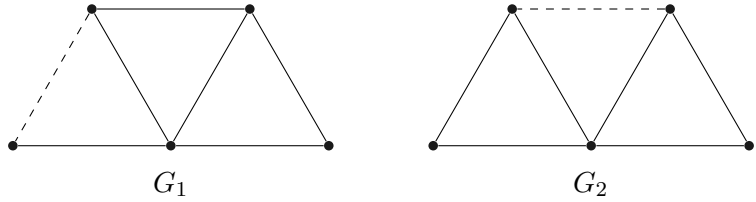

Additionally, the example in Figure~\ref{fig:cochromnonisomabs} shows that it is possible to have co-chromatic signed graphs $\Sigma_1$ and $\Sigma_2$ where both of them are unbalanced and the underlying graphs $|\Sigma_1|$ and $|\Sigma_2|$ are not isomorphic.
We have that
\[
\mathscr C(\Sigma_1, x) = \mathscr C(\Sigma_2, x) = \left( x (x-1) (x-2)^2 (x^2 - 3x +3), \; (x-1)^4 (x-2)^2 \right).
\]

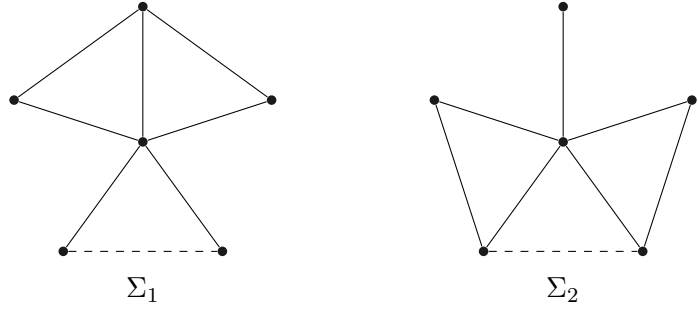
\begin{figure}[htbp]
    \centering
    \begin{subfigure}{0.33\textwidth}
        \centering
        \begin{tikzpicture}[scale=1.2, every node/.style={scale=0.9}]
        \tikzstyle{vertex}=[circle, thin, fill=black!90, inner sep=0pt, minimum width=4pt]
        \node [vertex] (v1) at (0, 0) {};
        \node [vertex] (v2) at (18:1.5) {};
        \node [vertex] (v3) at (90:1.5) {};
        \node [vertex] (v4) at (162:1.5) {};
        \node [vertex] (v5) at (234:1.5) {};
        \node [vertex] (v6) at (306:1.5) {};
        \draw  (v1) edge (v2);
        \draw  (v1) edge (v3);
        \draw  (v1) edge (v4);
        \draw  (v1) edge (v5);
        \draw  (v1) edge (v6);
        \draw  (v2) edge (v3);
        \draw  (v3) edge (v4);
        \draw  (v5) edge [dashed] (v6);
        \end{tikzpicture}
        \caption*{$\Sigma_1$}
    \end{subfigure}
    \begin{subfigure}{0.33\textwidth}
        \centering
        \begin{tikzpicture}[scale=1.2, every node/.style={scale=0.9}]
        \tikzstyle{vertex}=[circle, thin, fill=black!90, inner sep=0pt, minimum width=4pt]
        \node [vertex] (v1) at (0, 0) {};
        \node [vertex] (v2) at (18:1.5) {};
        \node [vertex] (v3) at (90:1.5) {};
        \node [vertex] (v4) at (162:1.5) {};
        \node [vertex] (v5) at (234:1.5) {};
        \node [vertex] (v6) at (306:1.5) {};
        \draw  (v1) edge (v2);
        \draw  (v1) edge (v3);
        \draw  (v1) edge (v4);
        \draw  (v1) edge (v5);
        \draw  (v1) edge (v6);
        \draw  (v2) edge (v6);
        \draw  (v4) edge (v5);
        \draw  (v5) edge [dashed] (v6);
        \end{tikzpicture}
        \caption*{$\Sigma_2$}
    \end{subfigure}
    \caption{Co-chromatic signed graphs $\Sigma_1$ and $\Sigma_2$.}
    \label{fig:cochromnonisomabs}
\end{figure}

Next, we conjecture the following:

\begin{conj} \label{conj:cochromsignedKn}
    There are no co-chromatic signed complete graphs.
\end{conj}

If Conjecture~\ref{conj:cochromsignedKn} is true, then the chromatic polynomials form a complete set of invariants for signed complete graphs up to switching isomorphism.
One can also consider versions of Conjecture~\ref{conj:cochromsignedKn} over other families of signed graphs, such as the signed threshold graphs that we will define in Section~\ref{sec:bivar}.

We conclude this section by revisiting the computations of chromatic polynomials of signed graphs in \cite{Involve}.
Beck et al.~\cite[Figure 1 and Figure 2]{Involve} list, up to switching isomorphism, all six signed graphs whose underlying graph is the Petersen graph and all signed graphs whose underlying graph is the complete graph $K_3$, $K_4$, or $K_5$.
They denote the signed Petersen graphs by $P_1, \dots, P_6$ and the signed complete graphs by $K_3^{(1)}$, $K_3^{(2)}$, $K_4^{(1)}$, $K_4^{(2)}$, $K_4^{(3)}$, $K_5^{(1)}, \dots, K_5^{(7)}$.
The chromatic polynomials of each of these signed graphs are listed in \cite[Theorem 1 and Theorem 2]{Involve}.
However, we find some inaccuracies in their computations of $\mathsf E(\Sigma, x)$ of the signed graphs $P_2$, $P_3$, $P_4$, $P_5$, $P_6$, $K_4^{(2)}$, $K_4^{(3)}$, $K_5^{(2)}$, $K_5^{(3)}$, $K_5^{(4)}$, $K_5^{(5)}$, $K_5^{(6)}$, and $K_5^{(7)}$.
We provide the chromatic polynomials of the signed Petersen graphs and the signed complete graphs in Table~\ref{tab:involvePetersen} and Table~\ref{tab:involvecomplete}, respectively.

\begin{table}[htbp]
\centering
\begin{tabularx}{\textwidth}{|c|X|X|}
\hline
$\Sigma$ & \multicolumn{1}{c|}{$\mathsf E(\Sigma, x)$} & \multicolumn{1}{c|}{$\mathsf O(\Sigma, x)$} \\
\hhline{|===|}
\hhline{|~~~|}
$P_1$ & $x(x-1)(x-2)(x^7-12x^6+67x^5-230x^4+529x^3-814x^2+775x-352)$ & $x(x-1)(x-2)(x^7-12x^6+67x^5-230x^4+529x^3-814x^2+775x-352)$ \\
\hline
$P_2$ & $x(x-2)(x^8-13x^7+79x^6-297x^5+763x^4-1379x^3+1717x^2-1351x+516)$ & $(x-1)^2(x-2)^2(x^6-9x^5+38x^4-98x^3+163x^2-165x+82)$ \\
\hline
$P_3$ & $x(x-2)(x^8-13x^7+79x^6-297x^5+765x^4-1397x^3+1781x^2-1462x+597)$ & $(x-1)(x^9-14x^8+91x^7-364x^6+995x^5-1938x^4+2703x^3-2621x^2+1619x-492)$ \\
\hline
$P_4$ & $x(x-2)(x^8-13x^7+79x^6-297x^5+767x^4-1411x^3+1823x^2-1524x+635)$ & $(x-1)(x^9-14x^8+91x^7-364x^6+997x^5-1956x^4+2773x^3-2767x^2+1781x-568)$ \\
\hline
$P_5$ & $x(x-2)(x^8-13x^7+79x^6-297x^5+765x^4-1401x^3+1803x^2-1509x+632)$ & $(x-1)(x-2)^2(x^2-4x+5)(x^5-6x^4+18x^3-34x^2+37x-28)$ \\
\hline
$P_6$ & $x(x^9-15x^8+105x^7-455x^6+1365x^5-2981x^4+4785x^3-5460x^2+4005x-1425)$ & $(x-1)(x^9-14x^8+91x^7-364x^6+1001x^5-1992x^4+2913x^3-3057x^2+2103x-727)$ \\
\hline
\end{tabularx}
\caption{The chromatic polynomials of the signed Petersen graphs, cf.~\cite[Figure $1$]{Involve}.}
\label{tab:involvePetersen}
\end{table}

\begin{table}[htbp]
\centering
\begin{tabular}{|c|l|l|}
\hline
$\Sigma$ & \multicolumn{1}{c|}{$\mathsf E(\Sigma, x)$} & \multicolumn{1}{c|}{$\mathsf O(\Sigma, x)$} \\
\hhline{|===|}
\hhline{|~~~|}
$K_3^{(1)}$ & $x(x-1)(x-2)$ & $x(x-1)(x-2)$ \\
\hline
$K_3^{(2)}$ & $x(x^2 - 3x + 3)$ & $(x-1)^3$ \\
\hhline{|=|=|=|}
$K_4^{(1)}$ & $x(x-1)(x-2)(x-3)$ & $x(x-1)(x-2)(x-3)$ \\
\hline
$K_4^{(2)}$ & $x(x-2)(x^2-4x+5)$ & $(x-1)^2 (x-2)^2$ \\
\hline
$K_4^{(3)}$ & $x(x^3-6x^2+15x-13)$ & $(x-1)(x^3-5x^2+10x-7)$ \\
\hhline{|=|=|=|}
$K_5^{(1)}$ & $x(x-1)(x-2)(x-3)(x-4)$ & $x(x-1)(x-2)(x-3)(x-4)$ \\
\hline
$K_5^{(2)}$ & $x(x-2)(x-3)(x^2-5x+7)$ & $(x-1)^2 (x-2) (x-3)^2$ \\
\hline
$K_5^{(3)}$ & $x(x-2)(x^3-8x^2+25x-29)$ & $(x-1)(x-2)(x^3-7x^2+18x-17)$ \\
\hline
$K_5^{(4)}$ & $x(x-2)(x-3)(x^2-5x+8)$ & $(x-1)(x-2)^3(x-3)$ \\
\hline
$K_5^{(5)}$ & $x(x-2)(x^3-8x^2+26x-31)$ & $(x-1)(x-2)(x^3-7x^2+19x-19)$ \\
\hline
$K_5^{(6)}$ & $x(x-2)(x-3)(x^2-5x+9)$ & $(x-1)(x-2)(x-3)(x^2-4x+5)$ \\
\hline
$K_5^{(7)}$ & $x(x^4-10x^3+45x^2-95x+75)$ & $(x-1)(x^4-9x^3+36x^2-69x+51)$ \\
\hline
\end{tabular}
\caption{The chromatic polynomials of the signed complete graphs, cf.~\cite[Figure $2$]{Involve}.}
\label{tab:involvecomplete}
\end{table}

\section{Joins of all-positive or all-negative signed complete graphs}
\label{sec:signedjoin}

In this section, we exhibit closed formulae for the chromatic polynomials of signed graphs obtained from various joins of all-positive or all-negative signed complete graphs.

Let $\Sigma_1$ and $\Sigma_2$ be signed graphs with disjoint vertex sets.
The following definitions are the same as those in \cite{Mattern21}.
The \textbf{all-positive join} of $\Sigma_1$ and $\Sigma_2$, denoted by $\Sigma_1 \veeplus \Sigma_2$, is the signed graph obtained by connecting each vertex of $\Sigma_1$ to each vertex of $\Sigma_2$ with a positive edge.
Similarly, the \textbf{all-negative join} of $\Sigma_1$ and $\Sigma_2$, denoted by $\Sigma_1 \veeminus \Sigma_2$, is the signed graph obtained by connecting each vertex of $\Sigma_1$ to each vertex of $\Sigma_2$ with a negative edge.
If $\Sigma$ is a signed graph, we let
\[
\Sigma \veeplus K_0 = K_0 \veeplus \Sigma = \Sigma \veeminus K_0 = K_0 \veeminus \Sigma = \Sigma.
\]
Note that the operations $\veeplus$ and $\veeminus$ both are associative and commutative.

Let $n \geqslant 0$ be an integer.
We define the polynomials
$$\displaystyle (x)_n \coloneqq \prod_{j=0}^{n-1} (x-j) \text{ and } \displaystyle \tensor*[_2]{(x)}{_n} \coloneqq \prod_{j=0}^{n-1} (x-2j).$$
As an example, we have $\mathsf E(+K_n, x) = \mathsf O(+K_n, x) = (x)_n$ for all integers $n \geqslant 0$.
Recall the Stirling number of the second kind, denoted by $\displaystyle \stirling{n}{k}$, which is the number of ways one can partition a set of $n$ elements into $k$ nonempty subsets \cite[Chapter $6$]{ConcreteMath}.
Let $n$ be an integer.
Define
\[
H(n,x) \coloneqq \begin{cases}
    \displaystyle
    \sum_{i=0}^n \stirling{n}{i} \cdot \tensor*[_2]{(x)}{_i}, & \text{if $n \geqslant 0$;} \\
    0, & \text{if $n<0$.}
\end{cases}
\]
We begin by stating a result of Zaslavsky.
\begin{prop}[{\cite[(5.7) and (5.8)]{ZasChromInv}}]
\label{prop:minKn}
Let $\Sigma$ be the signed graph $-K_n$ for some integer $n \geqslant 0$.
Then $\mathsf E(\Sigma, x) = H(n,x)$ and $\mathsf O(\Sigma, x) = \mathsf E(\Sigma, x-1) + n \cdot H(n-1, x-1)$.
\end{prop}

Next, we extend Proposition~\ref{prop:minKn} by obtaining closed formulae for signed graphs obtained by applying the operations $\veeplus$ or $\veeminus$ on all-positive or all-negative signed complete graphs.

Let $l, m, n$ be integers.
Define
\[
    H_1(l,m,n,x) \coloneqq \sum_{i=0}^l \sum_{j=0}^n \sum_{k=0}^{\min(i,j)} k! \binom{i}{k} \binom{j}{k} \stirling{l}{i} \stirling{n}{j}
    \cdot \tensor*[_2]{(x)}{_{i+j-k}} \cdot (x-i-j)_m
\]
if $l, m, n \geqslant 0$ and $H_1(l, m, n,x) \coloneqq 0$ if $l<0$, $m<0$, or $n<0$.
Let $A$ be a finite subset of $\mathbb{Z}$ and we define the set $-A \coloneqq \{-x : x\in A\}$.
The following proposition is a generalisation of Proposition~\ref{prop:minKn}.

\begin{prop}
\label{prop:lmn1}
    Let $\Sigma$ be the signed graph $(-K_l) \veeplus (+K_m) \veeplus (-K_n)$ for some integers $l, m, n \geqslant 0$.
    Then 
    \[
    \mathscr C(\Sigma, x) = \left( H_1(l,m,n,x), \; \mathsf E(\Sigma, x-1) + \hat H_1(l, m, n, x) \right),
    \]
    where $\hat H_1(l, m, n, x) \coloneqq l \cdot H_1(l-1, m, n, x-1) + m \cdot H_1(l, m-1, n, x-1) +n \cdot H_1(l, m, n-1, x-1)$.
\end{prop}

\begin{proof}
    Suppose that \( V(\left\vert \Sigma \right\vert) = V_1 \cup V_2 \cup V_3 \), where \( V_1 = V(\left\vert -K_l \right\vert) \), \( V_2 = V(\left\vert +K_m \right\vert) \), and \( V_3 = V(\left\vert -K_n \right\vert) \) are disjoint sets.
    Let $\lambda \geqslant 0$ be an integer and let \(\kappa : V(\left\vert \Sigma \right\vert) \to C_{\lambda}\) be a proper \(C_{\lambda}\)-colouring of \(\Sigma\).
    First, suppose that $\lambda$ is even.
    Due to the all-positive joins, the sets \(\kappa(V_1)\), \(\kappa(V_2)\), and \(\kappa(V_3)\) are pairwise disjoint. Since \(+K_m\) is all-positive, we have \(\left\vert \kappa(V_2) \right\vert = m\).
    Since both $-K_l$ and $-K_n$ are all-negative, we have $-\kappa(V_1) \cap \kappa(V_1) = -\kappa(V_3) \cap \kappa(V_3) = \varnothing$.
    Suppose that $\left\vert \kappa(V_1) \right\vert = i$ where $0 \leqslant i \leqslant l$ and $\left\vert \kappa(V_3) \right\vert = j$ where $0 \leqslant j \leqslant n$.
    Then $\kappa$ induces a partition of $V_1$ into $i$ disjoint nonempty subsets and a partition of $V_3$ into $j$ disjoint nonempty subsets.
    Let $S = \kappa(V_1) \cap -\kappa(V_3)$ and suppose that $\vert S \vert = k$ where $0 \leqslant k \leqslant \min(i,j)$.
    Then $-S = -\kappa(V_1) \cap \kappa(V_3)$.
    Next, we claim that
    \begin{align*}
        \mathsf E(\Sigma, \lambda) &= \sum_{i=0}^l \sum_{j=0}^n \sum_{k=0}^{\min(i,j)} k! \binom{i}{k} \binom{j}{k} \stirling{l}{i} \stirling{n}{j} \cdot \tensor*[_2]{(\lambda)}{_{i+j-k}} \cdot (\lambda-i-j)_m \\
        &= H_1(l, m, n, \lambda).
    \end{align*}
    Indeed, the factor $\displaystyle \stirling{l}{i} \stirling{n}{j}$ is the number of possible partitions of $V_1$ and $V_3$ as described above.
    The factor $\displaystyle k! \binom{i}{k} \binom{j}{k}$ is the number of choices for $S$ and $-S$ in $\kappa(V_1)$ and $\kappa(V_3)$.
    The factor $\tensor*[_2]{(\lambda)}{_{i+j-k}}$ is the number of ways to assign colours from $C_{\lambda}$ to $\left( \kappa(V_1) \cup \kappa(V_3) \right) \backslash W$ where $\left\vert \left( \kappa(V_1) \cup \kappa(V_3) \right) \backslash W \right\vert = i+j-k$.
    Lastly, since $\left\vert \kappa(V_1) \cup \kappa(V_3) \right\vert = i+j$, the factor $(\lambda-i-j)_m$ is the number of ways to assign the remaining available colours in $C_{\lambda}$ to $\kappa(V_2)$.
    Therefore, we conclude that $\mathsf E(\Sigma, x) = H_1(l,m,n,x)$.
    
    Suppose that $\lambda$ is odd.
    Note that there can be at most one vertex $v \in V(\left\vert \Sigma \right\vert)$ such that $\kappa(v) = 0$.
    The number of functions $\kappa$ such that $\kappa(v) \neq 0$ for all $v \in V(\left\vert \Sigma \right\vert)$ is equal to $\mathsf E(\Sigma, \lambda-1)$.
    Otherwise, suppose that $v$ is the only vertex of $\Sigma$ such that $\kappa(v) = 0$.
    The number of functions $\kappa$ such that $v \in V_1$ is equal to $l \cdot H_1(l-1, m, n, \lambda-1)$.
    The number of functions $\kappa$ such that $v \in V_2$ is equal to $m \cdot H_1(l, m-1, n, \lambda-1)$.
    The number of functions $\kappa$ such that $v \in V_3$ is equal to $n \cdot H_1(l, m, n-1, \lambda-1)$.
    Altogether, we conclude that $\mathsf O(\Sigma, x) = \mathsf E(\Sigma, x-1) + \hat H_1(l, m, n, x)$.
\end{proof}

Observe that $H_1(0,0,n,x) = H(n,x)$ for all integers $n$.
Thus, we can obtain Proposition~\ref{prop:minKn} by setting $l=m=0$ in Proposition~\ref{prop:lmn1}.

For all integers $m, n \geqslant 0$, we have
\begin{align*}
    H_1(0,m,n,x) = \sum_{j=0}^n \stirling{n}{j}
    \cdot \tensor*[_2]{(x)}{_j} \cdot (x-j)_m.
\end{align*}
Setting $l=0$ in Proposition~\ref{prop:lmn1} yields Corollary~\ref{cor:mn} below.

\begin{cor}
\label{cor:mn}
    Let $\Sigma$ be the signed graph $(+K_m) \veeplus (-K_n)$ for some integers $m, n \geqslant 0$.
    Then 
    \[
    \mathscr C(\Sigma, x) = \left( H_1(0,m,n,x), \; \mathsf E(\Sigma, x-1) + \hat H_1(0, m, n, x) \right),
    \]
    where $\hat H_1(0, m, n, x) = m \cdot H_1(0, m-1, n, x-1) +n \cdot H_1(0, m, n-1, x-1)$.
\end{cor}

Let $l, m, n$ be integers.
Define $H_2(l, m, n, x)$ to be 
\begin{align*}
     & \sum_{i=0}^l \sum_{j=0}^m \sum_{k=0}^n \sum_{s=0}^{\min(i,j)} \sum_{t=0}^k s! \binom{i}{s} \binom{j}{s} \binom{k}{t} \stirling{l}{i} \stirling{m}{j} \stirling{n}{k} \cdot \tensor*[_2]{(x)}{_{i+j+k-t-s}} \cdot (i+j-2s)_t
\end{align*}
if $l, m, n \geqslant 0$ and $H_2(l, m, n,x) \coloneqq 0$ if $l<0$, $m<0$, or $n<0$.

\begin{prop}
\label{prop:lmn2}
    Let $\Sigma$ be the signed graph $(-K_l) \veeplus (-K_m) \veeplus (-K_n)$ for some integers $l, m, n \geqslant 0$.
    Then 
    \[
    \mathscr C(\Sigma, x) = \left( H_2(l,m,n,x), \; \mathsf E(\Sigma, x-1) + \hat H_2(l, m, n, x) \right),
    \]
    where $\hat H_2(l, m, n, x) \coloneqq l \cdot H_2(l-1, m, n, x-1) + m \cdot H_2(l, m-1, n, x-1) + n \cdot H_2(l, m, n-1, x-1)$.
\end{prop}

\begin{proof}
    Suppose that $V (\left\vert \Sigma \right\vert) = V_1 \cup V_2 \cup V_3$ where $V_1 = V (\left\vert -K_l \right\vert)$, $V_2 = V (\left\vert -K_m \right\vert)$, and $V_3 = V (\left\vert -K_n \right\vert)$ are disjoint sets.
    Let $\lambda \geqslant 0$ be an integer and let \(\kappa : V (\left\vert \Sigma \right\vert) \to C_{\lambda}\) be a proper \(C_{\lambda}\)-colouring of \(\Sigma\).
    First, suppose that $\lambda$ is even.
    Due to the all-positive joins, the sets $\kappa(V_1)$, $\kappa(V_2)$, and $\kappa(V_3)$ are pairwise disjoint.
    Note that
    \(
    -\kappa(V_1) \cap \kappa(V_1) = -\kappa(V_2) \cap \kappa(V_2) = -\kappa(V_3) \cap \kappa(V_3) = \varnothing.
    \)
    Suppose that $\left\vert \kappa(V_1) \right\vert = i$ where $0 \leqslant i \leqslant l$, $\left\vert \kappa(V_2) \right\vert = j$ where $0 \leqslant j \leqslant m$, and $\left\vert \kappa(V_3) \right\vert = k$ where $0 \leqslant k \leqslant n$.
    Then $\kappa$ induces a partition of $V_1$ into $i$ disjoint nonempty subsets, a partition of $V_2$ into $j$ disjoint nonempty subsets, and a partition of $V_3$ into $k$ disjoint nonempty subsets.
    Let $S = \kappa(V_1) \cap -\kappa(V_2)$ and suppose that $\vert S \vert = s$ where $0 \leqslant s \leqslant \min(i,j)$.
    Then $-S = -\kappa(V_1) \cap \kappa(V_2)$.
    Let $W = \kappa(V_3) \cap \left( -\kappa(V_1) \cup -\kappa(V_2) \right)$ and suppose that $\left\vert W \right\vert = t$ where $0 \leqslant t \leqslant k$.
    The elements of $W$ are the additive inverses of some of the elements of $\kappa(V_1) \cup \kappa(V_2)$.
    Additionally, we also have that $W \cap \left(-S \cup S \right) = \varnothing$.
    Next, we claim that
    \begin{align*}
        \mathsf E(\Sigma, \lambda) =& \sum_{i=0}^l \sum_{j=0}^m \sum_{k=0}^n \sum_{s=0}^{\min(i,j)} \sum_{t=0}^k s! \binom{i}{s} \binom{j}{s} \binom{k}{t} \stirling{l}{i} \stirling{m}{j} \stirling{n}{k} \cdot \tensor*[_2]{(\lambda)}{_{i+j+k-t-s}} \cdot (i+j-2s)_t \\
        =& \; H_2(l, m, n, \lambda).
    \end{align*}
    Indeed, the factor $\displaystyle \stirling{l}{i} \stirling{m}{j} \stirling{n}{k}$ is the number of possible partitions of $V_1$, $V_2$, and $V_3$ as described above. 
    The factor $\displaystyle s! \binom{i}{s} \binom{j}{s}$ is the number of choices for $S$ and $-S$ in $\kappa(V_1)$ and $\kappa(V_2)$.
    The factor $\displaystyle \binom{k}{t}$ comes from choosing $t$ elements of $W$ from $\kappa(V_3)$.
    Since $\left\vert \left( -\kappa(V_1) \cup -\kappa(V_2) \right) \backslash \left( -S \cup S \right) \right\vert = i+j-2s$, the factor $(i+j-2s)_t$ is the number of ways to assign colours from $\left( -\kappa(V_1) \cup -\kappa(V_2) \right) \backslash \left( -S \cup S \right)$ to $W$.
    Lastly, the factor $\tensor*[_2]{(\lambda)}{_{i+j+k-t-s}}$ is the number of ways to assign colours from $C_{\lambda}$ to $\kappa(V_1) \cup \kappa(V_2) \cup \left( \kappa(V_3) \backslash W \right)$.
    Note that the choices of $\kappa(V_1) \cup \kappa(V_2)$ can be determined by the set $\left( \kappa(V_1) \backslash S \right) \cup \kappa(V_2)$.
    Therefore, we conclude that $\mathsf E(\Sigma, x) = H_2(l,m,n,x)$.
    The case where $\lambda$ is odd is similar to that of the proof of Proposition~\ref{prop:lmn1}.
\end{proof}

Let $n$ and $k$ be nonnegative integers such that $n \geqslant 2k$.
Denote by $T(n,k)$ the number of matchings of size $k$ in the unsigned complete graph $K_n$ \cite[Sequence A100861]{oeis}.
We have the formula $\displaystyle T(n,k) = \frac{(n)_{2k}}{2^k \cdot k!}$.

Let $l, m, n$ be integers.
Define $ H_3(l, m, n, x)$ to be
\begin{align*}
    & \sum_{i=0}^{\min(l,m)} \sum_{j=0}^{\left\lfloor\frac{l-i}{2} \right\rfloor} \sum_{k=0}^{\left\lfloor\frac{m-i}{2} \right\rfloor} i! \binom{l}{i} \binom{m}{i} \cdot T(l-i,j) \cdot T(m-i,k) \cdot \tensor*[_2]{(x)}{_{l+m-i-j-k}} \cdot (x-l-m+i)_n
\end{align*}
if $l, m, n \geqslant 0$ and $H_3(l, m, n,x) \coloneqq 0$ if $l<0$, $m<0$, or $n<0$.

\begin{prop}
\label{prop:lmn3}
    Let $\Sigma$ be the signed graph $\left( (+K_l) \veeminus (+K_m) \right) \veeplus (+K_n)$ for some nonnegative integers $l$, $m$, and $n$.
    Then 
    \[
    \mathscr C(\Sigma, x) = \left( H_3(l,m,n,x), \; \mathsf E(\Sigma, x-1) + \hat H_3(l, m, n, x) \right),
    \]
    where $\hat H_3(l, m, n, x) \coloneqq l \cdot H_3(l-1, m, n, x-1) + m \cdot H_3(l, m-1, n, x-1) + n \cdot H_3(l, m, n-1, x-1)$.
\end{prop}

\begin{proof}
    Suppose that $V (\left\vert \Sigma \right\vert) = V_1 \cup V_2 \cup V_3$ where $V_1 = V (\left\vert +K_l \right\vert)$, $V_2 = V (\left\vert +K_m \right\vert)$, and $V_3 = V (\left\vert +K_n \right\vert)$ are disjoint sets.
    Let $\lambda \geqslant 0$ be an integer and let \(\kappa : V (\left\vert \Sigma \right\vert) \to C_{\lambda}\) be a proper \(C_{\lambda}\)-colouring of \(\Sigma\).
    First, suppose that $\lambda$ is even.
    Due to the all-positive join, the sets $\kappa(V_1) \cup \kappa(V_2)$ and $\kappa(V_3)$ are disjoint.
    Moreover, the all-negative join implies that $-\kappa(V_1) \cap \kappa(V_2) = \kappa(V_1) \cap -\kappa(V_2) = \varnothing$.
    Note that $\left\vert \kappa(V_1) \right\vert = l$, $\left\vert \kappa(V_2) \right\vert = m$, and $\left\vert \kappa(V_3) \right\vert = n$.
    Let $S = \kappa(V_1) \cap \kappa(V_2)$ and suppose that $\vert S \vert = i$ where $0 \leq i \leq \min(l,m)$.
    Let $W_1 = -\kappa(V_1) \cap \kappa(V_1)$ and suppose that $\vert W_1 \vert = 2j$ where $0 \leqslant j \leqslant \lfloor \frac{l-i}{2} \rfloor$.
    Similarly, let $W_2 = -\kappa(V_2) \cap \kappa(V_2)$ and suppose that $\vert W_2 \vert = 2k$ where $0 \leqslant k \leqslant \lfloor \frac{m-i}{2} \rfloor$.
    We have that $-W_1 = W_1$, $-W_2 = W_2$, and $S \cap W_1 = S \cap W_2 = \varnothing$.
    Next, we claim that $\mathsf E(\Sigma, \lambda) = H_3(l, m, n, \lambda)$.
    Indeed, the factor $\displaystyle i! \binom{l}{i} \binom{m}{i}$ is the number of choices for $S$ in $\kappa(V_1)$ and $\kappa(V_2)$.
    The factor $T(l-i,j)$ is the number of ways to choose $W_1$ from $\kappa(V_1) \backslash S$, while $T(m-i,k)$ is the number of ways to choose $W_2$ from $\kappa(V_2) \backslash S$.
    The factor $\tensor*[_2]{(\lambda)}{_{l+m-i-j-k}}$ is the number of ways to assign colours from $C_{\lambda}$ to $\kappa(V_1)$ and $\kappa(V_2)$.
    The choices of $\kappa(V_1) \cup \kappa(V_2)$ can be determined by $S$, half of $W_1$, half of $W_2$, $\kappa(V_1) \backslash \left( S \cup W_1 \right)$, and $\kappa(V_2) \backslash \left( S \cup W_2 \right)$.
    Altogether, we have $l+m-i-j-k$ spots to allocate the colours from $C_{\lambda}$.
    Lastly, since $\left\vert \kappa(V_1) \cup \kappa(V_2) \right\vert = l+m-i$, the factor $(\lambda-l-m+i)_n$ is the number of ways to assign the remaining available colours in $C_{\lambda}$ to $\kappa(V_3)$.
    Therefore, we conclude that $\mathsf E(\Sigma, x) = H_3(l,m,n,x)$.
    The case where $\lambda$ is odd is similar to that of the proof of Proposition~\ref{prop:lmn1}.
\end{proof}

Let $i,j,k,l, m, n, s, t$ be integers such that $l+m+s-i-j-k \geqslant t \geqslant 0$.
Define 
\[
U_x(i,j,k,l,m,s,t) \coloneqq \tensor*[_2]{(x)}{_{l+m+s-i-j-k-t}} \cdot (l+m-i-2j-2k)_t
\]
and define $H_4(l, m, n, x)$ as 
\begin{align*}
    & \sum_{i=0}^{\min(l,m)}
    \sum_{j=0}^{\left\lfloor\frac{l-i}{2} \right\rfloor}
    \sum_{k=0}^{\left\lfloor\frac{m-i}{2} \right\rfloor}
    \sum_{s=0}^n \sum_{t=0}^s i! \binom{l}{i} \binom{m}{i} \binom{s}{t} \stirling{n}{s} \cdot T(l-i,j) \cdot T(m-i,k) \cdot U_{x}(i,j,k,l,m,s,t)
\end{align*}
if $l, m, n \geqslant 0$ and $H_4(l, m, n,x) \coloneqq 0$ if $l<0$, $m<0$, or $n<0$.

\begin{prop}
\label{prop:lmn4}
    Let $\Sigma$ be the signed graph $\left( (+K_l) \veeminus (+K_m) \right) \veeplus (-K_n)$ for some nonnegative integers $l$, $m$, and $n$.
    Then 
    \[
    \mathscr C(\Sigma, x) = \left( H_4(l,m,n,x), \; \mathsf E(\Sigma, x-1) + \hat H_4(l, m, n, x) \right),
    \]
    where $\hat H_4(l, m, n, x) \coloneqq l \cdot H_4(l-1, m, n, x-1) + m \cdot H_4(l, m-1, n, x-1) + n \cdot H_4(l, m, n-1, x-1)$.
\end{prop}

\begin{proof}
    Suppose that $V (\left\vert \Sigma \right\vert) = V_1 \cup V_2 \cup V_3$ where $V_1 = V (\left\vert +K_l \right\vert)$, $V_2 = V (\left\vert +K_m \right\vert)$, and $V_3 = V (\left\vert -K_n \right\vert)$ are disjoint sets.
    Let $\lambda \geqslant 0$ be an integer and let \(\kappa : V (\left\vert \Sigma \right\vert) \to C_{\lambda}\) be a proper \(C_{\lambda}\)-colouring of \(\Sigma\).
    First, suppose that $\lambda$ is even.
    Due to the all-positive join, the sets $\kappa(V_1) \cup \kappa(V_2)$ and $\kappa(V_3)$ are disjoint.
    Moreover, the all-negative join implies that $-\kappa(V_1) \cap \kappa(V_2) = \kappa(V_1) \cap -\kappa(V_2) = \varnothing$.
    Note that $\left\vert \kappa(V_1) \right\vert = l$ and $\left\vert \kappa(V_2) \right\vert = m$.
    Let $S = \kappa(V_1) \cap \kappa(V_2)$ and suppose that $\vert S \vert = i$ where $0 \leqslant i \leqslant \min(l,m)$.
    Let $W_1 = -\kappa(V_1) \cap \kappa(V_1)$ and suppose that $\vert W_1 \vert = 2j$ where $0 \leqslant j \leqslant \lfloor \frac{l-i}{2} \rfloor$.
    Similarly, let $W_2 = -\kappa(V_2) \cap \kappa(V_2)$ and suppose that $\vert W_2 \vert = 2k$ where $0 \leqslant k \leqslant \lfloor \frac{m-i}{2} \rfloor$.
    We have that $-W_1 = W_1$, $-W_2 = W_2$, and $S \cap W_1 = S \cap W_2 = \varnothing$.
    Next, we have $-\kappa(V_3) \cap \kappa(V_3) = \varnothing$ and suppose that $\left\vert \kappa(V_3) \right\vert = s$ where $0 \leqslant s \leqslant n$.
    Then $\kappa$ induces a partition of $V_3$ into $s$ disjoint nonempty subsets.
    Let $X = \kappa(V_3) \cap \left( -\kappa(V_1) \cup -\kappa(V_2) \right)$ and suppose that $\vert X \vert = t$ where $0 \leqslant t \leqslant s$.
    The elements of $X$ are the additive inverses of some of the elements of $\kappa(V_1) \cup \kappa(V_2)$.
    Additionally, we also have that $X \cap \left( W_1 \cup W_2 \right) = \varnothing$.
    Next, we claim that $\mathsf E(\Sigma, \lambda) = H_4(l, m, n, \lambda)$.
    Indeed, the factor $\displaystyle i! \binom{l}{i} \binom{m}{i}$ is the number of choices for $S$ in $\kappa(V_1)$ and $\kappa(V_2)$.
    The factor $T(l-i,j)$ is the number of ways to choose $W_1$ from $\kappa(V_1) \backslash S$, while $T(m-i,k)$ is the number of ways to choose $W_2$ from $\kappa(V_2) \backslash S$.
    The factor $\displaystyle \stirling{n}{s}$ is the number of possible partitions of $V_3$ as described above, and $\displaystyle \binom{s}{t}$ comes from choosing $t$ elements of $X$ from $\kappa(V_3)$.
    Since $$\left\vert \left( -\kappa(V_1) \cup -\kappa(V_2) \right) \backslash \left( W_1 \cup W_2 \right) \right\vert = l+m-i-2j-2k,$$ the factor $(l+m-i-2j-2k)_t$ is the number of ways to assign colours from the set $\left( -\kappa(V_1) \cup -\kappa(V_2) \right) \backslash \left( W_1 \cup W_2 \right)$ to $X$.
    Lastly, the factor $\tensor*[_2]{(\lambda)}{_{l+m+s-i-j-k-t}}$ is the number of ways to assign colours from $C_{\lambda}$ to $\kappa(V_1) \cup \kappa(V_2) \cup \left( \kappa(V_3) \backslash X \right)$.
    The choices of $\kappa(V_1) \cup \kappa(V_2)$ can be determined by $S$, half of $W_1$, half of $W_2$, $\kappa(V_1) \backslash \left( S \cup W_1 \right)$, and $\kappa(V_2) \backslash \left( S \cup W_2 \right)$.
    Altogether, we have $l+m+s-i-j-k-t$ spots to allocate the colours from $C_{\lambda}$.
    Therefore, we conclude that $\mathsf E(\Sigma, x) = H_4(l,m,n,x)$.
    The case where $\lambda$ is odd is similar to that of the proof of Proposition~\ref{prop:lmn1}.
\end{proof}

We can apply the propositions above to obtain various identities involving the functions $H$, $H_1$, $H_2$, $H_3$, and $H_4$.
We collect a sample of such identities below, which may be of independent interest.
Let $l$, $m$, and $n$ be nonnegative integers.

\begin{enumerate}[label=\roman*.]
    \item Since $(-K_l) \veeplus (+K_m) \veeplus (-K_n) = (-K_n) \veeplus (+K_m) \veeplus (-K_l)$, by Proposition~\ref{prop:lmn1}, we have that $H_1(l,m,n,x) = H_1(n,m,l,x)$.
    Furthermore, $H_1(l,m,n,x) = H_1(n,m,l,x)$ for all integers $l$, $m$, and $n$.
    
    \item Similarly, by Proposition~\ref{prop:lmn2}, we have that
    \begin{align*}
        & H_2(l,m,n,x) = H_2(l,n,m,x) = H_2(m,l,n,x) = \\
        & H_2(m,n,l,x) = H_2(n,l,m,x) = H_2(n,m,l,x)
    \end{align*}
    for all integers $l$, $m$, and $n$.

    \item Since $(-K_l) \veeplus (+K_1) \veeplus (-K_n) = (-K_l) \veeplus (-K_1) \veeplus (-K_n)$, by Proposition~\ref{prop:lmn1} and Proposition~\ref{prop:lmn2}, we have that $H_1(l,1,n,x) = H_2(l,1,n,x)$ for all integers $l$ and $n$. 
    
    \item Let $q$ and $r$ be nonnegative integers satisfying $q+r=n$.
    The signed graphs $-K_n$ and $(-K_q) \veeplus (-K_r)$ are switching isomorphic.
    Hence, by Proposition~\ref{prop:minKn}, Proposition~\ref{prop:lmn1}, and Proposition~\ref{prop:lmn2}, we have $H_1(q, 0, r, x) = H_2(q, 0, r, x) = H(n, x)$.
    In particular, we obtain the identity
    \begin{align*}
        \sum_{i=0}^n \stirling{n}{i} \cdot \tensor*[_2]{(x)}{_i} = \sum_{i=0}^q \sum_{k=0}^r \sum_{t=0}^k \binom{k}{t} \stirling{q}{i} \stirling{r}{k} \cdot \tensor*[_2]{(x)}{_{i+k-t}} \cdot (i)_t.
    \end{align*}
    
    \item Since $\left( (+K_l) \veeminus (+K_m) \right) \veeplus (-K_n) = \left( (+K_m) \veeminus (+K_l) \right) \veeplus (-K_n)$, by Proposition~\ref{prop:lmn4}, we have that $H_4(l,m,n,x) = H_4(m,l,n,x)$ for all integers $l$, $m$, and $n$.
    This equality can also be derived directly from the definition of $H_4$.
    
    \item Note that the signed graphs $\left( (+K_l) \veeminus (+K_m) \right) \veeplus (+K_n)$ and $\left( (+K_l) \veeminus (+K_n) \right) \veeplus (+K_m)$ are switching isomorphic.
    Therefore, by Proposition~\ref{prop:lmn3}, we have that
    \begin{align*}
        & H_3(l,m,n,x) = H_3(l,n,m,x) = H_3(m,l,n,x) =  \\
        & H_3(m,n,l,x) = H_3(n,l,m,x) = H_3(n,m,l,x)
    \end{align*}
    for all integers $l$, $m$, and $n$.

    \item Since $\left( (+K_l) \veeminus (+K_m) \right) \veeplus (+K_1) = \left( (+K_l) \veeminus (+K_m) \right) \veeplus (-K_1)$, by Proposition~\ref{prop:lmn3} and Proposition~\ref{prop:lmn4}, we have that $H_3(l,m,1,x) = H_4(l,m,1,x)$ for all integers $l$ and $m$.
    
    \item Let $q$ and $r$ be nonnegative integers satisfying $q+r=n$.
    The signed graph $+K_n$ is switching isomorphic to $(+K_q) \veeminus (+K_r)$.
    Hence, by Proposition~\ref{prop:lmn3} and Proposition~\ref{prop:lmn4}, we have $H_3(q,r,0,x) = H_4(q,r,0,x) = (x)_n$, which yields the following identity:
    \begin{align*}
        (x)_n = \sum_{i=0}^{\min(q,r)} \sum_{j=0}^{\left\lfloor\frac{q-i}{2} \right\rfloor} \sum_{k=0}^{\left\lfloor\frac{r-i}{2} \right\rfloor} i! \binom{q}{i} \binom{r}{i} \cdot T(q-i,j) \cdot T(r-i,k) \cdot \tensor*[_2]{(x)}{_{n-i-j-k}}.
    \end{align*}
    In particular, we have $\displaystyle (x)_n = \sum_{j=0}^{\lfloor n/2 \rfloor} T(n,j) \cdot \tensor*[_2]{(x)}{_{n-j}}$.
    
    \item Let $n \geqslant 1$.
    Observe that $(+K_m) \veeplus (-K_n)$ and $\left( (+K_1) \veeminus (+K_m) \right) \veeplus (-K_{n-1})$ are switching isomorphic.
    Therefore, by Corollary~\ref{cor:mn} and Proposition~\ref{prop:lmn4}, we deduce that $H_1(0,m,n,x) = H_4(1,m,n-1,x)$.
\end{enumerate}

\section{Bivariate chromatic polynomials and dominating-vertex deletion formulae}
\label{sec:bivar}

In this section, we define a generalisation of the chromatic polynomials of a signed graph, which we call the \emph{bivariate chromatic polynomials}.

\subsection{Bivariate chromatic polynomials}
\label{subsec:bivarchrom}

We begin by introducing the colouring of signed graphs in a more general setting.

\begin{dfn}
\label{dfn:bivarcolourset}
    Let $\lambda$ and $\mu$ be integers such that $\lambda \geqslant \mu \geqslant 0$ and let $C$ be a finite subset of $\mathbb{Z}$.
    We call $C$ a \textbf{$(\lambda, \mu)$-colour set} containing a \textbf{paired-colour set} $P$ and an \textbf{unpaired-colour set} $U$ if there exist disjoint subsets $P, U$ of $\mathbb{Z}^*$ such that $-P = P$, $|U| = \mu$, $-U \cap U = \varnothing$, and one of the following holds:
    \begin{enumerate}
        \item $\lambda-\mu$ is even, $|P| = \lambda-\mu$, and $C = P \cup U$.
        \item $\lambda-\mu$ is odd, $|P| = \lambda-\mu-1$, and $C = P \cup U \cup \{0\}$.
    \end{enumerate}
\end{dfn}

\begin{ex}
\label{ex:bivarcolourset}
    The empty set $\varnothing$ is the unique $(0,0)$-colour set and the singleton $\{0\}$ is the unique $(1,0)$-colour set.
    The set $\{-7, -6, -4, -1, 1, 3, 4, 7, 8\}$ is a $(9,3)$-colour set containing a paired-colour set $\{\pm 1, \pm 4, \pm 7\}$ and an unpaired-colour set $\{-6, 3, 8\}$.
    The set $\{-5, -2, -1, 0, 2, 3, 5\}$ is a $(7,2)$-colour set containing a paired-colour set $\{\pm 2, \pm 5\}$ and an unpaired-colour set $\{-1, 3\}$.
\end{ex}

Let $\Sigma = (\Gamma,\sigma)$ be a signed graph.
Given integers $\lambda$ and $\mu$ such that $\lambda \geqslant \mu \geqslant 0$, let $C_1$ be a $(\lambda, \mu)$-colour set containing a paired-colour set $P_1$ and an unpaired-colour set $U_1$.
Similarly, let $C_2$ be a $(\lambda, \mu)$-colour set containing a paired-colour set $P_2$ and an unpaired-colour set $U_2$.
We can construct a bijection $\alpha: C_1 \to C_2$ such that $\alpha(-c) = -\alpha(c)$ for all $c \in P_1$, $\alpha(U_1) = U_2$, and $\alpha(0) = 0$ if $\lambda-\mu$ is odd.
Observe that $\kappa$ is a proper $C_1$-colouring of $\Sigma$ if and only if $\alpha \circ \kappa$ is a proper $C_2$-colouring of $\Sigma$.
This implies that $f(\Sigma, \lambda, \mu)$ defined below depends only on $\Sigma$, $\lambda$, and $\mu$; it does not depend on the choice of $(\lambda, \mu)$-colour set.

\begin{dfn}
\label{dfn:bivarchrom}
    Let $\Sigma$ be a signed graph.
    Let $\lambda$ and $\mu$ be integers such that $\lambda \geqslant \mu \geqslant 0$ and let $C$ be a $(\lambda, \mu)$-colour set.
    Define $f(\Sigma, \lambda, \mu)$ to be the number of proper $C$-colourings of $\Sigma$.
\end{dfn}

Denote by $p(\Sigma)$ the number of all-positive connected components of $\Sigma$.
We can determine $f(\Sigma, \lambda, \mu)$ by using Theorem~\ref{thm:iechrombivar} below.

\begin{thm}
\label{thm:iechrombivar}
    Let $\Sigma = (\Gamma,\sigma)$ be a signed graph and let $\lambda$ and $\mu$ be integers such that $\lambda \geqslant \mu \geqslant 0$.
    Then
    \begin{equation*}
        f(\Sigma, \lambda, \mu) = 
            \displaystyle \sum_{i=0}^{\left\vert E(\Gamma) \right\vert} (-1)^i \sum_{\substack{Y \subseteq E(\Gamma) \vspace{0.7 pt} \\ |Y|=i}} \lambda^{p(\Sigma \vert Y)} \cdot (\lambda-\mu)^{b(\Sigma \vert Y)-p(\Sigma \vert Y)} \cdot \delta^{c(\Sigma \vert Y)-b(\Sigma \vert Y)}, 
    \end{equation*}
    where $\delta \in \{0,1\}$ is congruent to $\lambda-\mu$ modulo $2$.
\end{thm}
\begin{proof}
    Let $C$ be a $(\lambda, \mu)$-colour set containing unpaired-colour set $U$.
    Let $Y \subseteq E(\Gamma)$ and let $R_Y$ be the set of functions $\kappa : V(\Gamma) \to C$ such that $\kappa(v) = \sigma(\{v,w\}) \kappa(w)$ for all edges $\{v,w\} \in Y$.
    If $Y = \{e\}$, we may write $R_Y$ as $R_e$.
    First, we show that
    \[
    \left\vert R_Y \right\vert = \lambda^{p(\Sigma \vert Y)} \cdot (\lambda-\mu)^{b(\Sigma \vert Y)-p(\Sigma \vert Y)} \cdot \delta^{c(\Sigma \vert Y)-b(\Sigma \vert Y)},
    \]
    where $\delta \in \{0,1\}$ is congruent to $\lambda-\mu$ modulo $2$.
    Let $\kappa \in R_Y$ and let $\Lambda$ be a connected component of the signed graph $\Sigma \vert Y$.
    Let $\rho_{\Lambda}$ be the total number of possible functions $\kappa$ when restricted to the vertex set $V(\vert \Lambda \vert)$.
    We have that $\left\vert R_Y \right\vert = \prod \rho_{\Lambda}$ where the product is taken over all connected components of $\Sigma \vert Y$.
    Note that the function $\kappa$, when restricted to $V(\vert \Lambda \vert)$, is determined by its value on one vertex of $\Lambda$.
    If $\Lambda$ is all-positive, then $\rho_{\Lambda} = \lambda$.
    Suppose that $\Lambda$ is balanced and has at least one negative edge, say $\{v, w\}$.
    Since $\kappa(v) = -\kappa(w)$, we must have $\kappa(v) \in C \backslash U$.
    Hence, we conclude that $\rho_{\Lambda} = \lambda-\mu$ if $\Lambda$ is balanced and has at least one negative edge.
    There are $b(\Sigma \vert Y)-p(\Sigma \vert Y)$ of such $\Lambda$ in $\Sigma \vert Y$.
    Lastly, suppose that $\Lambda$ is unbalanced.
    Then $\rho_{\Lambda} = 0$ if $\lambda-\mu$ is even and $\rho_{\Lambda} = 1$ if $\lambda-\mu$ is odd.
    Combining all of these cases, we obtain the desired formula for $\left\vert R_Y \right\vert$.
    
    Observe that for all $Y, Z \subseteq E(\Gamma)$, we have $R_Y \cap R_Z = R_{Y \cup Z}$.
    Therefore, by the inclusion-exclusion principle, we have
    \begin{align*}
        f(\Sigma, \lambda, \mu) &= \lambda^{\left\vert V(\Gamma) \right\vert} - \left\vert \bigcup_{e \in E(\Gamma)} R_e \right\vert = \left\vert R_{\varnothing} \right\vert - \left( \sum_{i=1}^{|E(\Gamma)|} (-1)^{i+1} \sum_{Y \subseteq E(\Gamma), \, |Y|=i} \left\vert R_Y \right\vert \right) \\
        &= \sum_{i=0}^{\left\vert E(\Gamma) \right\vert} (-1)^i \sum_{Y \subseteq E(\Gamma), \, |Y|=i} \left\vert R_Y \right\vert,
    \end{align*}
    as required.
\end{proof}

Clearly, for each integer $\mu \geqslant 0$, there exist infinitely many integers $\lambda \geqslant \mu$ such that $\lambda-\mu$ is even and infinitely many integers $\lambda \geqslant \mu$ such that $\lambda-\mu$ is odd.
Then, by Theorem~\ref{thm:iechrombivar}, we conclude that there exist unique $\mathsf E(\Sigma, x, y)$, $\mathsf O(\Sigma, x, y) \in \mathbb{Z}[x,y]$ such that for all integers $\lambda$ and $\mu$ where $\lambda \geqslant \mu \geqslant 0$, we have $f(\Sigma, \lambda, \mu) = \mathsf E(\Sigma, \lambda, \mu)$ if $\lambda-\mu$ is even, and $f(\Sigma, \lambda, \mu) = \mathsf O(\Sigma, \lambda, \mu)$ if $\lambda-\mu$ is odd.
We call $\mathsf E(\Sigma, x, y)$ the \textbf{even bivariate chromatic polynomial} of $\Sigma$ and $\mathsf O(\Sigma, x, y)$ the \textbf{odd bivariate chromatic polynomial} of $\Sigma$.
For any signed graph $\Sigma$, define $\mathscr B(\Sigma, x, y) \coloneqq \left( \mathsf E(\Sigma, x, y), \; \mathsf O(\Sigma, x, y) \right)$, which we refer to as the \textbf{bivariate chromatic polynomials} of $\Sigma$.
Given signed graphs $\Sigma_1$ and $\Sigma_2$, it follows that $\mathscr B(\Sigma_1, x, y) = \mathscr B(\Sigma_2, x, y)$ if and only if $\mathsf E(\Sigma_1, x, y) = \mathsf E(\Sigma_2, x, y)$ and $\mathsf O(\Sigma_1, x, y) = \mathsf O(\Sigma_2, x, y)$.
Equivalently, we also have that $\mathscr B(\Sigma_1, x, y) = \mathscr B(\Sigma_2, x, y)$ if and only if for all integers $\lambda$ and $\mu$ such that $\lambda \geqslant \mu \geqslant 0$, we have $f(\Sigma_1, \lambda, \mu) = f(\Sigma_2, \lambda, \mu)$.

Let $\lambda \geqslant 0$ be an integer.
Then the $\lambda$-colour set $C_{\lambda}$ is also a $(\lambda, 0)$-colour set.
It follows that $f(\Sigma, \lambda, 0) = f(\Sigma, \lambda)$ and hence, we obtain the following corollary.

\begin{cor}
\label{cor:bivaru0}
    Let $\Sigma$ be a signed graph.
    Then $\mathscr B(\Sigma, x, 0) = \mathscr C(\Sigma, x)$.
\end{cor}

It follows from the next corollary that, for any signed graph $\Sigma$, $\mathscr B(\Sigma, x, y) = \mathscr C(\Sigma, x)$ if and only if $\Sigma$ is all-positive.

\begin{cor}
\label{cor:negedges}
    Let $\Sigma = (\Gamma, \sigma)$ be a signed graph and let $\nu$ be the number of negative edges of $\Sigma$.
    Then the coefficients of $x^{\left\vert V(\Gamma) \right\vert-1}$ and $x^{\left\vert V(\Gamma) \right\vert-2} y$ in both $\mathsf E(\Sigma, x, y)$ and $\mathsf O(\Sigma, x, y)$ are equal to $-\left\vert E(\Gamma) \right\vert$ and $\nu$, respectively.
    Moreover, if $\Sigma$ is all-positive, then $\mathscr B(\Sigma, x, y) = \mathscr C(\Sigma, x) = \left( \chi(\Gamma, x), \; \chi(\Gamma, x) \right)$.
\end{cor}
\begin{proof}
    In Theorem~\ref{thm:iechrombivar}, observe that if $\vert Y \vert = i > 1$, then $c(\Sigma \vert Y) \leqslant \left\vert V(\Gamma) \right\vert-2$.
    Hence, the coefficients of $x^{\left\vert V(\Gamma) \right\vert-1}$ and $x^{\left\vert V(\Gamma) \right\vert-2} y$ in both $\mathsf E(\Sigma, x, y)$ and $\mathsf O(\Sigma, x, y)$ are obtained precisely when $\vert Y \vert = i = 1$.
    Suppose that $\vert Y \vert = 1$ and thus, the signed graph $\Sigma \vert Y$ consists of $\left\vert V(\Gamma) \right\vert-2$ copies of $K_1$ and either $+K_2$ or $-K_2$.
    Applying Theorem~\ref{thm:iechrombivar}, the following sum yields the coefficients of $x^{\left\vert V(\Gamma) \right\vert-1}$ and $x^{\left\vert V(\Gamma) \right\vert-2} y$ in both $\mathsf E(\Sigma, x, y)$ and $\mathsf O(\Sigma, x, y)$:
    \[
    -(\left\vert E(\Gamma) \right\vert - \nu) x^{\left\vert V(\Gamma) \right\vert - 1} - \nu x^{\left\vert V(\Gamma) \right\vert - 2} (x-y) = -\left\vert E(\Gamma) \right\vert x^{\left\vert V(\Gamma) \right\vert - 1} + \nu x^{\left\vert V(\Gamma) \right\vert - 2} y.
    \]
    Next, suppose that $\Sigma$ is all-positive.
    By definition, for all integers $\lambda$ and $\mu$ such that $\lambda \geqslant \mu \geqslant 0$, we have $f(\Sigma, \lambda, \mu) = f(\Gamma, \lambda) = f(\Sigma, \lambda)$.
    Therefore, we conclude that $\mathscr B(\Sigma, x, y) = \mathscr C(\Sigma, x) = \left( \chi(\Gamma, x), \; \chi(\Gamma, x) \right)$.
\end{proof}

We find that $\mathscr B(\Sigma, x, y)$ is not always preserved under switching.
For example, we have $\mathscr B(+K_2,x,y) = \left( x(x-1), \; x(x-1) \right)$ while $\mathscr B(-K_2,x,y) = \left( x^2-x+y, \; x^2-x+y \right)$.
Meanwhile, if two signed graphs $\Sigma_1$ and $\Sigma_2$ are isomorphic, then for all integers $\lambda$ and $\mu$ such that $\lambda \geqslant \mu \geqslant 0$, we clearly have $f(\Sigma_1, \lambda, \mu) = f(\Sigma_2, \lambda, \mu)$.

\begin{prop}
\label{prop:isombivarchrom}
    Let $\Sigma_1$ and $\Sigma_2$ be signed graphs.
    If $\Sigma_1 \cong \Sigma_2$ then $\mathscr B(\Sigma_1, x, y) = \mathscr B(\Sigma_2, x, y)$.
\end{prop}

Similar to Theorem~\ref{thm:polyiffbal}, the following theorem states that the even and odd bivariate chromatic polynomials are equal precisely when $\Sigma$ is balanced.

\begin{thm}
\label{thm:bivarpolyiffbal}
    Let $\Sigma$ be a signed graph.
    Then $\mathsf E(\Sigma,x,y) = \mathsf O(\Sigma,x,y)$ if and only if $\Sigma$ is balanced.
\end{thm}
\begin{proof}
    Suppose that $\Sigma = (\Gamma, \sigma)$ is balanced and let $Y \subseteq E(\Gamma)$.
    Then $\Sigma \vert Y$ is also balanced and hence, $c(\Sigma \vert Y) = b(\Sigma \vert Y)$.
    By Theorem~\ref{thm:iechrombivar}, we conclude that
    \[
        \mathsf E(\Sigma,x,y) = \mathsf O(\Sigma,x,y) = \sum_{i=0}^{|E(\Gamma)|} (-1)^i \sum_{Y \subseteq E(\Gamma), \, |Y|=i} x^{p(\Sigma \vert Y)} \cdot (x-y)^{b(\Sigma \vert Y)-p(\Sigma \vert Y)}.
    \]
    Next, suppose that $\Sigma$ is unbalanced but $\mathsf E(\Sigma,x,y) = \mathsf O(\Sigma,x,y)$.
    By Corollary~\ref{cor:bivaru0}, we have that
    \(
    \mathsf E(\Sigma,x) = \mathsf E(\Sigma,x,0) = \mathsf O(\Sigma,x,0) = \mathsf O(\Sigma,x),
    \)
    which contradicts Theorem~\ref{thm:polyiffbal}.
    Therefore, if $\Sigma$ is unbalanced, then $\mathsf E(\Sigma,x,y) \neq \mathsf O(\Sigma,x,y)$.
\end{proof}

Here we also compute the even and odd bivariate chromatic polynomials of the signed graphs $G_1$ and $G_2$ in Figure~\ref{fig:signedgem}.
We find that both $\mathscr B(G_1,x,y)$ and $\mathscr B(G_2,x,y)$ are equal to 
\[
\left( (x-2)^2 (x^3-3x^2+xy +3x-2y), \; (x-2)^2 (x^3-3x^2+xy +3x-2y-1) \right).
\]
For signed graphs $\Sigma_1$ and $\Sigma_2$ in Figure~\ref{fig:cochromnonisomabs}, we have that both $\mathscr B(\Sigma_1,x,y)$ and $\mathscr B(\Sigma_2,x,y)$ are equal to
\[
\left( (x-1) (x-2)^2 (x^3-3x^2+xy +3x-2y), \; (x-1) (x-2)^2 (x^3-3x^2+xy +3x-2y-1) \right).
\]

In general, given any signed graph $\Sigma$, we cannot simply replace $\mathscr B(\Sigma,x,y)$ by either one of $\mathsf E(\Sigma,x,y)$ or $\mathsf O(\Sigma,x,y)$.
As an example, we have two unbalanced signed graphs $\Sigma_3$ and $\Sigma_4$ in Figure~\ref{fig:sigma3sigma4} where both $\mathsf O(\Sigma_3,x,y)$ and $\mathsf O(\Sigma_4,x,y)$ are equal to
\[
x^5 - 7x^4 + 3x^3y + 20x^3 - 15x^2y - 29x^2 + xy^2 + 26xy + 21x -2y^2-15y-6,
\]
while
\begin{align*}
    \mathsf E(\Sigma_3,x,y) &= x^5 - 7x^4 + 3x^3y + 20x^3 - 15x^2y - 28x^2 + xy^2 + 26xy + 16x - 2y^2 - 15y,\\
    \mathsf E(\Sigma_4,x,y) &= x^5 - 7x^4 + 3x^3y + 20x^3 - 15x^2y - 27x^2 + xy^2 + 26xy + 14x - 2y^2 - 14y
\end{align*}
are distinct.
We immediately have that $\mathsf O(\Sigma_3,x) = \mathsf O(\Sigma_4,x) = (x-1)^2(x-2)(x^2-3x+3)$, while $\mathsf E(\Sigma_3,x) = x(x-2)^2(x^2-3x+4)$ and $\mathsf E(\Sigma_4,x) = x(x-2)(x^3-5x^2+10x-7)$ are distinct.
The pair $\Sigma_3$ and $\Sigma_4$ is not an isolated occurrence, and it illustrates the importance of considering both the even and odd chromatic polynomials, whether univariate or bivariate.

\begin{figure}[htbp]
    \centering
    \begin{subfigure}{0.33\textwidth}
        \centering
        \begin{tikzpicture}[scale=1.2, every node/.style={scale=0.9}]
        \tikzstyle{vertex}=[circle, thin, fill=black!90, inner sep=0pt, minimum width=4pt]
        \node [vertex] (v2) at (18:1.5) {};
        \node [vertex] (v1) at (90:1.5) {};
        \node [vertex] (v5) at (162:1.5) {};
        \node [vertex] (v4) at (234:1.5) {};
        \node [vertex] (v3) at (306:1.5) {};
        \draw  (v1) edge (v2);
        \draw  (v1) edge (v5);
        \draw  (v2) edge (v3);
        \draw  (v2) edge [dashed] (v4);
        \draw  (v3) edge (v4);
        \draw  (v3) edge [dashed] (v5);
        \draw  (v4) edge [dashed] (v5);
        \end{tikzpicture}
        \caption*{$\Sigma_3$}
    \end{subfigure}
    \begin{subfigure}{0.33\textwidth}
        \centering
        \begin{tikzpicture}[scale=1.2, every node/.style={scale=0.9}]
        \tikzstyle{vertex}=[circle, thin, fill=black!90, inner sep=0pt, minimum width=4pt]
        \node [vertex] (v2) at (18:1.5) {};
        \node [vertex] (v1) at (90:1.5) {};
        \node [vertex] (v5) at (162:1.5) {};
        \node [vertex] (v4) at (234:1.5) {};
        \node [vertex] (v3) at (306:1.5) {};
        \draw  (v1) edge (v2);
        \draw  (v1) edge [dashed] (v3);
        \draw  (v1) edge (v4);
        \draw  (v1) edge [dashed] (v5);
        \draw  (v2) edge (v3);
        \draw  (v3) edge (v4);
        \draw  (v4) edge [dashed] (v5);
        \end{tikzpicture}
        \caption*{$\Sigma_4$}
    \end{subfigure}
    \caption{Signed graphs $\Sigma_3$ and $\Sigma_4$.}
    \label{fig:sigma3sigma4}
\end{figure}
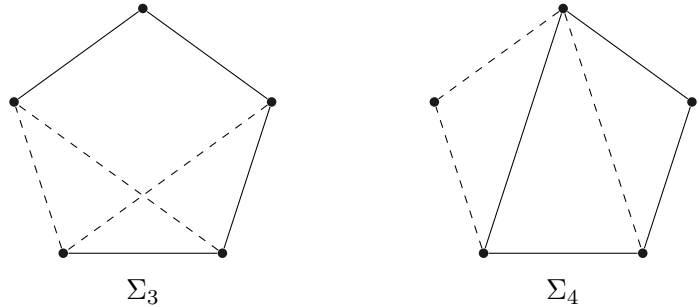

\subsection{Dominating-vertex deletion formulae}
\label{subsec:deletionform}

Let $\Sigma$ be a signed graph and let $v$ be a vertex of $\Sigma$.
We call $v$ a \textbf{positive dominating vertex} if $v$ is connected with positive edges to all other vertices of $\Sigma$, \textbf{negative dominating vertex} if $v$ is connected with negative edges to all other vertices of $\Sigma$, and \textbf{isolated vertex} if $v$ is not connected to any other vertices of $\Sigma$.
By these definitions, we let the only vertex of $K_1$ to be simultaneously an isolated vertex, positive dominating vertex, and negative dominating vertex of $K_1$.
We utilise the bivariate chromatic polynomials to find recursive dominating-vertex deletion formulae.
For isolated vertex, the recursive deletion formula is straightforward.
Let $\Sigma \backslash v$ denote the signed graph obtained from $\Sigma$ by removing the vertex $v$ and all edges that are incident with $v$.
If $v$ is an isolated vertex of $\Sigma$, then $\mathsf E(\Sigma,x,y) = x \cdot \mathsf E(\Sigma \backslash v,x,y)$ and $\mathsf O(\Sigma,x,y) = x \cdot \mathsf O(\Sigma \backslash v,x,y)$.

\begin{prop}
\label{prop:posuniv}
    Let $\Sigma$ be a signed graph and let $v$ be a vertex of $\Sigma$.
    If $v$ is a positive dominating vertex of $\Sigma$ then 
    \begin{align*}
    \mathsf E(\Sigma,x,y) &= y \cdot \mathsf E(\Sigma \backslash v,x-1,y-1) + (x-y) \cdot \mathsf E(\Sigma \backslash v,x-1,y+1), \\
    \mathsf O(\Sigma,x,y) &= y \cdot \mathsf O(\Sigma \backslash v,x-1,y-1) + (x-y-1) \cdot \mathsf O(\Sigma \backslash v,x-1,y+1) + \mathsf E(\Sigma \backslash v,x-1,y).
    \end{align*}
\end{prop}
\begin{proof}
    Suppose that $v$ is a positive dominating vertex of $\Sigma$.
    It is easy to check that both equations hold when $\Sigma$ is the signed graph $K_1$.
    We will then prove both equations combinatorially.
    Let $\lambda$ and $\mu$ be integers such that $\lambda \geqslant \mu \geqslant 0$.
    Let $C$ be a $(\lambda, \mu)$-colour set containing a paired-colour set $P$ and an unpaired-colour set $U$.
    Let $\kappa$ be a proper $C$-colouring and let $D = C \backslash \{\kappa(v)\}$.
    For the first equation, it suffices to show that
    \begin{align*}
        \mathsf E(\Sigma,\lambda,\mu) = \mu \cdot \mathsf E(\Sigma \backslash v,\lambda-1,\mu-1) + (\lambda-\mu) \cdot \mathsf E(\Sigma \backslash v,\lambda-1,\mu+1)
    \end{align*}
    whenever $\lambda - 1 \geqslant \mu-1 \geqslant 0$ and $\lambda - 1 \geqslant \mu+1$, or equivalently, whenever $\mu \geqslant 1$ and $\lambda-\mu \geqslant 2$.
    Suppose that $\lambda-\mu$ is even.
    First, suppose that $\kappa(v) \in U$.
    Note that $D$ is a $(\lambda-1, \mu-1)$-colour set so we want to find the number of proper $D$-colourings of $\Sigma \backslash v$, which is equal to $\mathsf E(\Sigma \backslash v,\lambda-1,\mu-1)$.
    Hence, the number of $\kappa$ such that $\kappa(v) \in U$ is equal to $\mu \cdot \mathsf E(\Sigma \backslash v,\lambda-1,\mu-1)$.
    If $\kappa(v) \in P$ then $D$ is a $(\lambda-1, \mu+1)$-colour set.
    Hence, the number of $\kappa$ such that $\kappa(v) \in P$ is equal to $(\lambda-\mu) \cdot \mathsf E(\Sigma \backslash v,\lambda-1,\mu+1)$.
    Altogether, we obtain the first equation.    

    For the second equation, it suffices to show that if $\mu \geqslant 1$ and $\lambda-\mu \geqslant 2$, then
    \begin{align*}
        \mathsf O(\Sigma,\lambda,\mu) = \mu \cdot \mathsf O(\Sigma \backslash v,\lambda-1,\mu-1) + (\lambda-\mu-1) \cdot \mathsf O(\Sigma \backslash v,\lambda-1,\mu+1) + \mathsf E(\Sigma \backslash v,\lambda-1,\mu).
    \end{align*}
    Suppose that $\lambda-\mu$ is odd.
    If $\kappa(v) \in U$ then $D$ is a $(\lambda-1, \mu-1)$-colour set.
    If $\kappa(v) \in P$ then $D$ is a $(\lambda-1, \mu+1)$-colour set.
    Lastly, if $\kappa(v) = 0$ then $D$ is a $(\lambda-1, \mu)$-colour set.
    Altogether, we obtain the second equation.
\end{proof}

The negative dominating-vertex deletion formulae can be proved in a similar manner, as we now show.

\begin{prop}
\label{prop:neguniv}
    Let $\Sigma$ be a signed graph and let $v$ be a vertex of $\Sigma$.
    If $v$ is a negative dominating vertex of $\Sigma$ then
    \begin{align*}
        \mathsf E(\Sigma,x,y) &= y \cdot \mathsf E(\Sigma \backslash v,x,y) + (x-y) \cdot \mathsf E(\Sigma \backslash v,x-1,y+1), \\
        \mathsf O(\Sigma,x,y) &= y \cdot \mathsf O(\Sigma \backslash v,x,y) + (x-y-1) \cdot \mathsf O(\Sigma \backslash v,x-1,y+1) + \mathsf E(\Sigma \backslash v,x-1,y).
    \end{align*}
\end{prop}
\begin{proof}
    Suppose that $v$ is a negative dominating vertex of $\Sigma$.
    It is easy to check that both equations hold when $\Sigma$ is the signed graph $K_1$.
    We will then prove both equations combinatorially.
    Let $\lambda$ and $\mu$ be integers such that $\lambda \geqslant \mu \geqslant 0$.
    Let $C$ be a $(\lambda, \mu)$-colour set containing a paired-colour set $P$ and an unpaired-colour set $U$.
    Let $\kappa$ be a proper $C$-colouring and let $D = C \backslash \{-\kappa(v)\}$.
    For the first equation, it suffices to show that if $\lambda-\mu \geqslant 2$, then
    \begin{align*}
        \mathsf E(\Sigma,\lambda,\mu) = \mu \cdot \mathsf E(\Sigma \backslash v,\lambda,\mu) + (\lambda-\mu) \cdot \mathsf E(\Sigma \backslash v,\lambda-1,\mu+1).
    \end{align*}
    Suppose that $\lambda-\mu$ is even.
    If $\kappa(v) \in U$ then $D=C$ so the number of such $\kappa$ is equal to $\mu \cdot \mathsf E(\Sigma \backslash v,\lambda,\mu)$.
    If $\kappa(v) \in P$ then $D$ is a $(\lambda-1, \mu+1)$-colour set.
    Hence, the number of $\kappa$ such that $\kappa(v) \in P$ is equal to $(\lambda-\mu) \cdot \mathsf E(\Sigma \backslash v,\lambda-1,\mu+1)$.
    Altogether, we obtain the first equation.

    For the second equation, it suffices to show that if $\lambda-\mu \geqslant 2$, then
        \begin{align*}
            \mathsf O(\Sigma,\lambda,\mu) = \mu \cdot \mathsf O(\Sigma \backslash v,\lambda,\mu) + (\lambda-\mu-1) \cdot \mathsf O(\Sigma \backslash v,\lambda-1,\mu+1) + \mathsf E(\Sigma \backslash v,\lambda-1,\mu).
        \end{align*}
    Suppose that $\lambda-\mu$ is odd.
    If $\kappa(v) \in U$ then $D=C$.
    If $\kappa(v) \in P$ then $D$ is a $(\lambda-1, \mu+1)$-colour set.
    Lastly, if $\kappa(v) = 0$ then $D$ is a $(\lambda-1, \mu)$-colour set.
    Altogether, we obtain the second equation.
\end{proof}

Let $\Sigma$ be a signed graph and let $v$ be a positive or negative dominating vertex of $\Sigma$.
If $\Sigma \backslash v$ is balanced, then, by Theorem~\ref{thm:bivarpolyiffbal}, we have $\mathsf E(\Sigma \backslash v,x,y) = \mathsf O(\Sigma \backslash v,x,y)$.
Applying Propositions~\ref{prop:posuniv} and \ref{prop:neguniv}, we obtain a relation between $\mathsf E(\Sigma,x,y)$ and $\mathsf O(\Sigma,x,y)$ below.

\begin{cor}
\label{cor:balancedcone}
    Let $\Sigma$ be a signed graph and let $v$ be a positive or negative dominating vertex of $\Sigma$.
    If the signed graph $\Sigma \backslash v$ is balanced then 
    \[
    \mathsf E(\Sigma,x,y) - \mathsf O(\Sigma,x,y) = \mathsf E(\Sigma \backslash v,x-1,y+1) - \mathsf E(\Sigma \backslash v,x-1,y).
    \]
\end{cor}

\subsection{Signed threshold graphs}
\label{subsec:signedthresh}

A signed graph is called a \textbf{signed threshold graph} if it is the signed graph $K_1$ or it can be obtained from $K_1$ by repeatedly adding an isolated vertex, positive dominating vertex, or negative dominating vertex.
Let $n \geqslant 0$ be an integer and let $\bm{a} =(a_1,a_2,\dots,a_n) \in \{-1,0,1\}^n$.
In particular, let $() \in \{-1,0,1\}^0$ denote the empty tuple.
We denote by $\mathcal{T}_{\bm{a}}$ the signed threshold graph that can be constructed as follows: we start with $K_1$ and then from $i=1$ to $i=n$, we add an isolated vertex if $a_i = 0$, positive dominating vertex if $a_i = 1$, and negative dominating vertex if $a_i = -1$.
In particular, we have that $\mathcal{T}_{()} = K_1$.
Observe that the underlying graphs of signed threshold graphs are the unsigned \emph{threshold graphs} \cite{MP95}.

Compared with Theorem~\ref{thm:iechrombivar}, Propositions~\ref{prop:posuniv} and \ref{prop:neguniv} allow us to compute the even and odd bivariate chromatic polynomials of signed threshold graphs much more efficiently via use of the recursive deletion formulae.
Consequently, the even and odd chromatic polynomials of signed threshold graphs can also be computed more efficiently using these formulae rather than directly using Theorem~\ref{thm:iechrom}.

\begin{ex}
\label{ex:signedthresh}
    Let $\bm{a} = (1, -1, 0, -1, 1, 0, 1)$.
    Then
    \begin{align*}
        \mathsf E(\mathcal{T}_{\bm{a}},x,y) =& \, (x-1)(x-3)(x^6-15x^5+6x^4y+99x^4-71x^3 y -345x^3+4x^2 y^2+325x^2 y \\
        & +618x^2-36x y^2-650x y -440x+80y^2+424y),\\
        \mathsf O(\mathcal{T}_{\bm{a}},x,y) =& \, (x-1)(x-3)(x^6-15x^5+6x^4y+99x^4-71x^3 y -359x^3+4x^2 y^2+325x^2 y \\
        & +738x^2-36x y^2-666x y-792x+80y^2+488y+328).
    \end{align*}
    Thus, we also have that
    \begin{align*}
        \mathsf E(\mathcal{T}_{\bm{a}},x) &= x(x-1)(x-2)(x-3)(x^4-13x^3+73x^2-199x+220), \\
        \mathsf O(\mathcal{T}_{\bm{a}},x) &= (x-1)^2(x-3)(x^5-14x^4+85x^3-274x^2+464x-328).
    \end{align*}
\end{ex}

The following lemma provides a relation between the even bivariate chromatic polynomial of a signed graph $\Sigma$ and the chromatic polynomials of $\Sigma^+$.

\begin{lem}
\label{lem:lambdauequal}
    Let $\Sigma$ be a signed graph.
    Then $\mathsf E(\Sigma,x,x) = \mathsf E(\Sigma^+,x) = \mathsf O(\Sigma^+,x)$.
\end{lem}

\begin{proof}
    Denote the underlying graph of $\Sigma^+$ as $\Gamma^+$.
    Since $\Sigma^+$ is all-positive, by Corollary~\ref{cor:negedges}, we have $\mathsf E(\Sigma^+,x) = \mathsf O(\Sigma^+,x) = \chi(\Gamma^+,x)$.
    We will then prove that $\mathsf E(\Sigma,x,x) = \chi(\Gamma^+,x)$ by showing that $f(\Sigma,\lambda,\lambda) = f(\Gamma^+,\lambda)$ for all integers $\lambda \geqslant 0$.
    Let $\lambda \geqslant 0$ be an integer and let $C$ be a $(\lambda, \lambda)$-colour set containing a paired-colour set $P$ and an unpaired-colour set $U$.
    It follows that $P = \varnothing$, $C = U$, and $\vert C \vert = \lambda$.
    Suppose that $V = V(|\Sigma|) = V(\Gamma^+)$ and let $\kappa : V \to C$ be a function.
    If $\kappa$ is a proper $C$-colouring of $\Sigma$, then $\kappa$ is also a proper $C$-colouring of $\Gamma^+$.
    Conversely, suppose that $\kappa$ is a proper $C$-colouring of $\Gamma^+$.
    Let $\{v, w\}$ be a negative edge of $\Sigma$ so clearly, $\{v, w\} \notin E(\Gamma^+)$.
    Since $C = U$, we have $\kappa(v) \neq -\kappa(w)$.
    This implies that $\kappa$ is also a proper $C$-colouring of $\Sigma$.
    Therefore, we have proved that $\kappa$ is a proper $C$-colouring of $\Sigma$ if and only if $\kappa$ is a proper $C$-colouring of $\Gamma^+$.
    By definition, it follows that $f(\Sigma,\lambda,\lambda) = f(\Gamma^+,\lambda)$, as desired.
\end{proof}

In the next theorem, equality between the signed threshold graphs refers to their equality as isomorphism classes or unlabeled signed graphs.

\begin{thm}
\label{thm:signthresh}
    Let $n \geqslant 0$ be an integer and let $\bm{a}, \bm{b} \in \{0,1\}^n \cup \{\pm 1\}^n$.
    Then
    \begin{align*}
        \bm{a} = \bm{b} \iff \mathcal{T}_{\bm{a}} = \mathcal{T}_{\bm{b}} \iff \mathscr B(\mathcal{T}_{\bm{a}},x,y) = \mathscr B(\mathcal{T}_{\bm{b}},x,y) \iff \mathsf E(\mathcal{T}_{\bm{a}},x,y) = \mathsf E(\mathcal{T}_{\bm{b}},x,y).
    \end{align*}
\end{thm}

\begin{proof}
    It suffices to prove that if $\mathsf E(\mathcal{T}_{\bm{a}},x,y) = \mathsf E(\mathcal{T}_{\bm{b}},x,y)$ then $\bm{a} = \bm{b}$, or equivalently, if $\bm{a} \neq \bm{b}$ then $\mathsf E(\mathcal{T}_{\bm{a}},x,y) \neq \mathsf E(\mathcal{T}_{\bm{b}},x,y)$.
    First, suppose that $\bm{a} \neq \bm{b}$ where $\bm{a}, \bm{b} \in \{0,1\}^n$.
    Then, both $\mathcal{T}_{\bm{a}}$ and $\mathcal{T}_{\bm{b}}$ are all-positive.
    Let $\bm{a} = (a_1,\dots,a_n)$ and $\bm{b} = (b_1,\dots,b_n)$.
    Furthermore, let $\bm{s} = (s_1,\dots,s_n)$ and $\bm{t} = (t_1,\dots,t_n)$ such that $\displaystyle s_i = \sum_{j=0}^{i-1} a_{n-j}$ and $\displaystyle t_i = \sum_{j=0}^{i-1} b_{n-j}$ for all $i \in \{1,\dots,n\}$.
    Note that the entries of both $\bm{s}$ and $\bm{t}$ monotonically increase as the indices increase.
    By Corollary~\ref{cor:negedges} and \cite[Theorem 2.1]{CM19} \footnote{We warn the reader of a typo in the conclusion of Theorem 2.1 in \cite{CM19}.}, we have 
    \begin{align*}
        \mathsf E(\mathcal{T}_{\bm{a}},x,y) = \chi(\left\vert \mathcal{T}_{\bm{a}} \right\vert,x) = x \prod_{i=1}^n (x - s_i) \text{   and   } \mathsf E(\mathcal{T}_{\bm{b}},x,y) = \chi(\left\vert \mathcal{T}_{\bm{b}} \right\vert,x) = x \prod_{i=1}^n (x - t_i).
    \end{align*}
    Since $\bm{a} \neq \bm{b}$, we have that $\bm{s} \neq \bm{t}$ and therefore, $\mathsf E(\mathcal{T}_{\bm{a}},x,y) \neq \mathsf E(\mathcal{T}_{\bm{b}},x,y)$.

    Next, suppose that $\mathsf E(\mathcal{T}_{\bm{a}},x,y) = \mathsf E(\mathcal{T}_{\bm{b}},x,y)$ where $\bm{a}, \bm{b} \in \{\pm 1\}^n$.
    In particular, we have $\mathsf E(\mathcal{T}_{\bm{a}},x,x) = \mathsf E(\mathcal{T}_{\bm{b}},x,x)$.
    By Lemma~\ref{lem:lambdauequal}, we obtain $\mathsf E(\mathcal{T}_{\bm{a}}^+,x) = \mathsf E(\mathcal{T}_{\bm{b}}^+,x)$.
    Let $\bm{a}^+$, $\bm{b}^+ \in \{0,1\}^n$ be the tuples obtained from $\bm{a}$ and $\bm{b}$, respectively, by replacing all instances of $-1$ with $0$.
    It follows that $\mathcal{T}_{\bm{a}}^+ = \mathcal{T}_{\bm{a}^+}$ and $\mathcal{T}_{\bm{b}}^+ = \mathcal{T}_{\bm{b}^+}$, which means that $\mathsf E(\mathcal{T}_{\bm{a}^+},x) = \mathsf E(\mathcal{T}_{\bm{b}^+},x)$.
    By Corollary~\ref{cor:negedges}, we have that $\mathsf E(\mathcal{T}_{\bm{a}^+},x,y) = \mathsf E(\mathcal{T}_{\bm{a}^+},x) = \mathsf E(\mathcal{T}_{\bm{b}^+},x) = \mathsf E(\mathcal{T}_{\bm{b}^+},x,y)$.
    Since $\bm{a}^+$, $\bm{b}^+ \in \{0,1\}^n$, we conclude that $\bm{a}^+ = \bm{b}^+$ as argued in the previous case above.
    This immediately implies that $\bm{a} = \bm{b}$.
    
    Lastly, suppose that $\bm{a} \neq \bm{b}$ where, without loss of generality, we assume that $\bm{a} \in \{0,1\}^n$ and $\bm{b} \in \{\pm 1\}^n$.
    If $\bm{b} \in \{1\}^n$ then we can apply the same argument as in the first case above.
    Otherwise, observe that $\mathcal{T}_{\bm{a}}$ is all-positive while $\mathcal{T}_{\bm{b}}$ has at least one negative edge.
    By Corollary~\ref{cor:negedges}, we have $\mathsf E(\mathcal{T}_{\bm{a}},x,y) = \mathsf E(\mathcal{T}_{\bm{a}},x) \in \mathbb{Z}[x]$ while the coefficient of $x^{n-1}y$ in $\mathsf E(\mathcal{T}_{\bm{b}},x,y)$ is nonzero.
    Therefore, we conclude that $\mathsf E(\mathcal{T}_{\bm{a}},x,y) \neq \mathsf E(\mathcal{T}_{\bm{b}},x,y)$.
    This completes the proof.
\end{proof}

We conjecture that for all integers $n \geqslant 0$, the set $\{0,1\}^n \cup \{\pm 1\}^n$ in the assumption of Theorem~\ref{thm:signthresh} can be extended to $\{-1,0,1\}^n$.
Applying the formulae in Propositions~\ref{prop:posuniv} and \ref{prop:neguniv}, we have computationally verified that Conjecture~\ref{conj:signedthresh} below is true up to $n =12$.

\begin{conj}
\label{conj:signedthresh}
    Let $n \geqslant 0$ be an integer and let $\bm{a}, \bm{b} \in \{-1,0,1\}^n$.
    Then
    \begin{align*}
        \bm{a} = \bm{b} \iff \mathcal{T}_{\bm{a}} = \mathcal{T}_{\bm{b}} \iff \mathscr B(\mathcal{T}_{\bm{a}},x,y) = \mathscr B(\mathcal{T}_{\bm{b}},x,y) \iff \mathsf E(\mathcal{T}_{\bm{a}},x,y) = \mathsf E(\mathcal{T}_{\bm{b}},x,y).
    \end{align*}
\end{conj}

Notably, if non-isomorphic signed complete graphs correspond uniquely to its bivariate chromatic polynomials, then non-isomorphic graphs can be associated uniquely with the bivariate chromatic polynomials as well.

\begin{conj}
\label{conj:isomiffbivar}
    Let $\Sigma_1$ and $\Sigma_2$ be signed complete graphs.
    Then
    \begin{align*}
        \Sigma_1 \cong \Sigma_2 \iff \mathscr B(\Sigma_1,x,y) = \mathscr B(\Sigma_2,x,y) \iff \mathsf E(\Sigma_1,x,y) = \mathsf E(\Sigma_2,x,y).
    \end{align*}
\end{conj}

If $\bm{b} \in \{\pm 1\}^n$ then $\mathcal{T}_{\bm{b}}$ is a signed complete graph.
By Theorem~\ref{thm:signthresh}, we obtain a partial result towards Conjecture~\ref{conj:isomiffbivar}: the conclusion of Conjecture~\ref{conj:isomiffbivar} is true if $\Sigma_1$ and $\Sigma_2$ are signed complete graphs that are also signed threshold graphs.
Additionally, we have computationally verified that Conjecture~\ref{conj:isomiffbivar} is true over all signed complete graphs on up to $7$ vertices.

Let $\Gamma_1^+$ be a graph such that $\Gamma_1^+ = \left\vert \Sigma_1^+ \right\vert$ for some signed complete graph $\Sigma_1$.
Similarly, let $\Gamma_2^+$ be a graph such that $\Gamma_2^+ = \left\vert \Sigma_2^+ \right\vert$ for some signed complete graph $\Sigma_2$.
If Conjecture~\ref{conj:isomiffbivar} is true, then $\Gamma_1^+$ and $\Gamma_2^+$ are isomorphic if and only if $\mathsf E(\Sigma_1,x,y) = \mathsf E(\Sigma_2,x,y)$.
Therefore, the even bivariate chromatic polynomial is a complete algebraic invariant of graphs, provided that Conjecture~\ref{conj:isomiffbivar} is true.

\section{Acknowledgement}

We are grateful to Thomas Zaslavsky for his comments on an earlier draft of this paper.

The first author was partially supported by the Singapore Ministry of Education Academic Research Fund; grant numbers: RG18/23 (Tier 1) and MOE-T2EP20222-0005 (Tier 2).

\bibliographystyle{amsplain}
\providecommand{\bysame}{\leavevmode\hbox to3em{\hrulefill}\thinspace}
\providecommand{\MR}{\relax\ifhmode\unskip\space\fi MR }
\providecommand{\MRhref}[2]{%
  \href{http://www.ams.org/mathscinet-getitem?mr=#1}{#2}
}
\providecommand{\href}[2]{#2}

\end{document}